\documentclass[a4paper,12pt]{article}
\usepackage{setspace}
\usepackage{fullpage}
\usepackage{mathtools}
\usepackage[utf8]{inputenc}
\usepackage{amsthm,amsmath,amstext, amssymb, xcolor, scrextend, tikz, multirow, enumerate, wasysym, makecell, tabularx, hyperref,  pbox, lipsum,csquotes,  tikz-cd, multicol, calc, accents, stackengine, kpfonts}
\usepackage[all]{xy}
\usepackage{xcolor}
\usepackage{palatino}
\usepackage{cancel}
\usepackage{authblk}

\usepackage{mathtools}

\DeclareSymbolFont{AMSb}{U}{msb}{m}{n}
\DeclareMathSymbol{\N}{\mathbin}{AMSb}{"4E}
\DeclareMathSymbol{\Z}{\mathbin}{AMSb}{"5A}
\DeclareMathSymbol{\R}{\mathbin}{AMSb}{"52}
\DeclareMathSymbol{\Q}{\mathbin}{AMSb}{"51}
\DeclareMathSymbol{\I}{\mathbin}{AMSb}{"49}
\DeclareMathSymbol{\C}{\mathbin}{AMSb}{"43}

\DeclareFontFamily{U}{mathx}{\hyphenchar\font45}
\DeclareFontShape{U}{mathx}{m}{n}{<-> mathx10}{}
\DeclareSymbolFont{mathx}{U}{mathx}{m}{n}
\DeclareMathAccent{\widebar}{0}{mathx}{"73}

\def\lim{\mathop{\rm lim}\nolimits}

\def\Hom{\mathrm{Hom}}

\newcommand{\tn}{\textnormal}
\newcommand{\F}{\mathbb{F}}

\newcommand{\subjclass}[2][2020]{%
  \noindent\textbf{\textit{2010 Mathematics Subject Classification}}%
  \textbf{ (#1)}: #2\par
}

\newcommand{\kb}{\mathfrak{b}}

\newcommand{\kc}{\mathfrak{c}}

\newcommand{\Zzeta}{\mathbb{Z}[\zeta_p]}

\newcommand{\Zzetab}{\mathbb{Z}[\zeta_{p^2}]}

\newcommand{\Zpzeta}{\mathbb{Z}_p[\zeta_p]}

\newcommand{\Zpbzetab}{\mathbb{Z}_p[\zeta_{p^2}]}

\newcommand{\iso}{\cong}

\numberwithin{equation}{section}
\newtheorem{theorem}{Theorem}
\newtheorem{prop}[theorem]{Proposition}
\newtheorem{lemma}[theorem]{Lemma}

\newtheorem{corol}[theorem]{Corollary}

\theoremstyle{remark}
\newtheorem{remark}[theorem]{REMARK}
\theoremstyle{definition}
\newtheorem{defn}[theorem]{Definition}
\newtheorem{example}[theorem]{Example}

\usetikzlibrary{decorations.shapes}
\usetikzlibrary{decorations.text}
\usetikzlibrary{calc}

\hypersetup{colorlinks=true,linkcolor=blue,citecolor=blue, filecolor=blue,urlcolor=blue}

\title{The profinite genus of the groups $\Z^n\rtimes C_{p^2}$}
\author[1]{ Marlon Estanislau}
 \author[2]{John W.\ MacQuarrie}
\author[3]{
Anderson Porto}
\affil[1]{Universidade Federal de Minas Gerais, Belo Horizonte, MG, Brazil}
\affil[2]{Universidade Federal de Minas Gerais, Belo Horizonte, MG, Brazil}

\affil[3]{Universidade Federal dos Vales do Jequitinhonha e Mucuri, Diamantina, MG, Brazil}
\begin{document}
\maketitle
\onehalfspacing
\footnotetext[1]{\textit{Email addresses:} mestanislau@ufmg.br (M.  Estanislau), john@mat.ufmg.br (J.W. MacQuarrie), ander.porto@ict.ufvjm.edu.br (A.L.P.  Porto).}
\footnotetext[2]{This research was undertaken during a postdoctoral project of the third author and a doctoral project of the first author, both under the supervision of the second author at the Federal University of Minas Gerais.}

\begin{abstract}
    A formula is given for the profinite genus of groups of the form 
$\mathbb{Z}^n \rtimes C_{p^2}$, completing the calculation  of the size of the genus  of semidirect products of the form $\mathbb{Z}^n \rtimes G$ where  $G$ is a finite $p$-group  of finite integral representation type.
\end{abstract}

\begin{keywords} {Profinite genus, semidirect products, integral representations, crystallographic groups, Galois actions.} \end{keywords}
\subjclass[2020]{Primary 20E34, 20C05, 20C10, 20H15. }

\section{Introduction}

A group $G$ is residually finite if every element of $G$ different from $1$ is different from $1$ in some finite quotient of $G$.  Recently, a much studied question is to what extent groups in given classes of residually finite groups of a combinatorial nature may be distinguished from one other by their sets of finite quotient groups. The study  started in the 1970's when Baumslag \cite{BAU74},  Stebe \cite{Ste72} and others found examples of distinct finitely generated residually finite groups having the same set of finite quotients. The study lead to the notion of the genus $\mathfrak{g}(G)$ of a residually finite group $G$: the set of isomorphism classes of finitely generated residually finite groups having the same set of finite quotients as those of $G$. Equivalently, $\mathfrak{g}(G)$ is the set of isomorphism classes of finitely generated residually finite groups whose profinite completion is isomorphic to the profinite completion $\widehat G$ of $G$. 
%Denote by $\mathfrak{RF}$ the class of all finitely generated residually finite groups. 
Calculating the exact cardinality $|\mathfrak{g}(G)|$ of the genus (which we will abusively also call the genus) is generally too difficult in practice, so it is common to ask whether it is finite, or $1$ --- groups $G$ for which $|\mathfrak{g}(G)| = 1$ are said to be ``profinitely rigid'' \cite{BCR16,bridson21, Paolini, Sam2, Sam,     BPZ, BPZ2, BZ, GZ,  Paolini2,  popovic}. There are only few papers where the exact value, or estimates, of the genus appear (see \cite{BZG, BPZ, BPZ2, BZ, GZ, Ner19, Ner20, Ner24}).  

One case where precise formulas have been achieved is in work of Grunewald and Zalesskii \cite[Proposition 2.23]{GZ}, where a formula for the genus of semidirect products $\Z^n\rtimes C_p$ is given, using Reiner's   classification \cite{rein0} of the $\Z C_p$-lattices (that is, finitely generated $\Z$-free $\Z C_p$-modules). Key to the success of \cite{GZ} is that $C_p$ has finite integral representation type (meaning there are a finite number of indecomposable $\Z C_p$-lattices).  There is only one other non-trivial finite $p$-group $G$ of finite integral representation type: the cyclic group of order $p^2$.  The characterization of $\Z C_{p^2}$-lattices is much more complicated than the characterization of $\Z C_p$-lattices \cite[Theorems 7.3 and 7.8]{reiner1978}, but it is explicit and hence, as noted in \cite{GZ}, a formula for the genus of groups of the form $\Z^n\rtimes C_{p^2}$ is possible, completing the calculation of the genus for groups $\Z^n\rtimes G$ where $G$ is a $p$-group of finite integral representation type.  We present the formula here.  Our main technical tool may be interesting in its own right: we define natural actions of the Galois groups $G(p^2):=\tn{Gal}(\Q(\zeta_{p^2}):\Q)$ and $G(p) := \tn{Gal}(\Q(\zeta_p):\Q)$ (where $\zeta_n$ denotes a primitive $n$th root of unity) on the set of isomorphism classes of $\Z C_{p^2}$-lattices, which allow us to make precise the relationship between (the number of isomorphism classes of) $\Z C_{p^2}$-lattices of $\Z$-rank $n$ and (the number of isomorphism classes of) groups of the form $\Z^n\rtimes C_{p^2}$.

The groups $\Z^n\rtimes C_{p^2}$ are examples of \emph{crystallographic} (or more precisely, \emph{symmorphic}) groups (see for instance \cite{Sam} for definitions) and thus our investigation is part of the wider study of the genus of crystallographic groups -- see %\cite[Proposition 1.2]%{Sam} for much of the state-of-the-art on the subject 
%also 
\cite{Paolini, Sam, Finken, Paolini2, popovic} and references therein for other work in this area.

In order to state our main result, we require the following notation.  For $i\geqslant 1$, let $\Z[\zeta_{p^i}]$ be the ring of algebraic integers of $\Q[\zeta_{p^i}].$ Let $H(\mathbb{Z}[\zeta_{p^i}])=H(p^i)$ be the ideal  class group of $\mathbb{Z}[\zeta_{p^i}]$ and let $h(\mathbb{Z}[\zeta_{p^i}])=h_i$ be its order (see for instance Chapter 1 of \cite{neukirch} for definitions). The group $G(p^i)$ acts naturally on the ideal class group $H(\mathbb{Z}[\zeta_{p^i}])$ (see Chapter 10 in \cite{wash}). The space of orbits of this action will be denoted by $G(p^i) \backslash \, H(\mathbb{Z}[\zeta_{p^i}])$ and its size will be written as $\left| G(p^i) \backslash \, H(\mathbb{Z}[\zeta_{p^i}]) \right|.$  For the  definition of the space of orbits $ G(p^2) \backslash \, U_t $  see Section \ref{s3}, and for $\Sigma_M$ see Definition \ref{sigma}. %Denote by $\mathfrak{RF}$ the class of all finitely generated residually finite groups. 
By ``faithful semidirect product'' we mean a semidirect product $\mathcal{E}=M \rtimes_{\psi_{}} C_{p^2}$, where $M$ is a finitely generated free abelian group and the action of $C_{p^2}$ on $M$ is faithful. The modules of types $(A)-(F)$ and the modules $E(\mathbb{Z}[\zeta_{p}])$, $E(\mathbb{Z}[\zeta_{p^2}])$ are described in Theorem \ref{integral indecom}.
 Our main result is as follows: 

 \begin{theorem} \label{caso geral}
Let $\mathcal{E}=\Z^n \rtimes_{\psi_{}} C_{p^2}$ be a faithful semidirect product. Then $$\left| G(p^2) \backslash \, H(\mathbb{Z}[\zeta_{p^2}])  \right| \leqslant |\mathfrak{g}\left(\mathcal{E}\right)| \leqslant 2 \cdot \left| G(p) \backslash \, H(\mathbb{Z}[\zeta_{p}])  \right| \cdot \left| G(p^2) \backslash \, H(\mathbb{Z}[\zeta_{p^2}])  \right| \cdot \left| G(p^2) \backslash \, U_t  \right|,$$
and $|\mathfrak{g}\left(\mathcal{E}\right)|$ 
can assume the following values: 
\begin{itemize}
\item $\left| G(p^2) \backslash \, H(\mathbb{Z}[\zeta_{p^2}])  \right|$, if $M_{}=\Z^{a} \oplus \displaystyle \bigoplus_{i=1}^{c} \, \mathfrak{c}_i  \oplus \displaystyle \bigoplus_{j=1}^{e} \, E(\mathfrak{c}'_j),$  
with $c+e \geqslant 1$ and $\mathfrak{c}_i,$ $E(\mathfrak{c}'_j)$ in the genus of $\mathbb{Z}[\zeta_{p^2}]$ and $E(\mathbb{Z}[\zeta_{p^2}])$, respectively;

%\item $\left| G(p^2) \backslash \, H(\mathbb{Z}[\zeta_{p^2}])  \right|$, if $M_{}$ only has  direct summands of type $(A)$ that are in the genus of $\Z$ or $\mathbb{Z}[\zeta_{p^2}]$ or in the genus of $E(\mathbb{Z}[\zeta_{p^2}])$;
  
  \item  $\left| G(p) \backslash \, H(\mathbb{Z}[\zeta_{p}])  \right| \cdot \left| G(p^2) \backslash \, H(\mathbb{Z}[\zeta_{p^2}])  \right|\Sigma_M$, if $M_{}$  has at least one direct summand in the genus of $\mathbb{Z}[\zeta_{p}]$ or $E(\mathbb{Z}[\zeta_{p}])$,  and at least one direct summand in the genus of $\mathbb{Z}[\zeta_{p^2}]$ or $E(\mathbb{Z}[\zeta_{p^2}])$;
  \item $\left| G(p) \backslash \, H(\mathbb{Z}[\zeta_{p}])  \right| \cdot \left| G(p^2) \backslash \, H(\mathbb{Z}[\zeta_{p^2}])  \right| \cdot \left| G(p^2) \backslash \, U_t  \right|\Sigma_M$, if M has no direct summands of type $(A)$ other than possibly $\Z$, and at least one direct summand of type $(B)-(F)$.
\end{itemize}
\end{theorem}

Theorem \ref{caso geral} completes the calculation of $|\mathfrak{g}( \Z^n \rtimes_{\psi_{}} C_{p^2})|$, because any faithful semidirect product falls into one of the above cases, while a non-faithful semidirect product $\Z^n\rtimes_{\psi} C_{p^2}$ has the same genus as the corresponding semidirect product $\Z^n\rtimes_{\psi}{(C_{p^2}/C_p)}$ by Theorem \ref{theorem sd products in same genus iff modules in same orbit}, and hence its genus is given by \cite[Proposition 2.23]{GZ}.

\subsubsection*{Acknowledgements}
     
    The first author was supported by  CAPES Doctoral Grant 88887.688170/2022-00 and by FAPEMIG Doctoral Grant 13632/2025-00.  The second author was partially supported by the CNPq Universal Grant 402934/2021-0, CNPq Produtividade 1D Grant 303667/2022-2, and FAPEMIG Universal Grants APQ-00971-22 and APQ-03491-25. The third author is grateful for the financial support from FAPDF and the Post-Doctoral internship period at the Federal University of Minas Gerais. We thank Sam Hughes for drawing our attention to relevant literature.

\section{Classification of the indecomposable $\mathbb{Z} G$-lattices and their profinite completions}\label{2}

%The following result lists all non-isomorphic  indecomposable $\mathbb{Z}C_{p^2}$-lattices (see for example Theorem 7.3 in \cite{reiner1978} or Theorem 1 in \cite{jones} or Theorem 34.35 in \cite{curtis}). When $G=\left \langle g \right \rangle \cong C_{p^2},$ denote by $\lambda$ the element $g - 1 \in \Z G.$ If $A$ is an associative ring with $1$, we will denote the group of units as $A^{\cdot}$ or $u(A).$ 

The classification of indecomposable $\mathbb{Z} C_{p^2}$-lattices due to Reiner (Theorem \ref{integral indecom}) requires some notation, which will be used throughout the paper.  We mostly maintain the notation used by Reiner and we present it here.

Let $G=\langle g\rangle\;$ be a cyclic group of order $p^2.$ Whenever $M$ is a $\Z G$-module we denote by $\overline{M}$ the reduction modulo $p$ of $M$, that is $\overline{M}:=M/pM$. Let 
$$\phi_p(x)=1+ x+\cdots+x^{p-1}\hbox{ and }\phi_{p^2}(x)=1+ x^p+\cdots+x^{p\cdot (p-1)}$$
%$\phi_p, \phi_{p^2}$ 
denote the $p$-th and $p^2$-th cyclotomic polynomials, respectively. % Denoting by $\zeta_{p^n}$ a primitive $p^n$th root of unity, we have $\phi_p(\zeta_p)=0$ and $\phi_{p^2}(\zeta_{p^2})=0$.

 We will denote by $\Lambda$ the ring $\Z G$, and we define three fundamental objects, $E$, $R$ and $S$.  For each of them, we will need to interpret them both as submodules of $\Lambda$, and as quotient rings of $\Lambda$, as follows:
 \begin{itemize}
  \item $E$ is the quotient ring $\Lambda/(g^p-1)\Lambda$.  It is isomorphic as a $\Lambda$-module to the submodule $\phi_{p^2}(g) \Lambda$ of $\langle g^p\rangle$-fixed points of $\Lambda$, which we also refer to as $E$.
  
  %Consider  the quotient ring $ \Lambda/(g^p-1)\Lambda$. We have $\Lambda/(g^p-1)\Lambda\simeq\phi_{p^2}(g) \Lambda$, where $\phi_{p^2}(g) \Lambda$ is the submodule of the $\langle g^p\rangle$-fixed points   of $\Lambda$. We will use these submodules extensively in the text, so we will use the notation $E$ to mean that we are referring to $\Lambda/(g^p-1)\Lambda$ or $\phi_{p^2}(g) \Lambda$ depending on the context; if we want to treat $E$  as quotient ring of $\Lambda$ or to treat $E$ as a submodule of $\Lambda$.
  \item $S$ is the quotient ring $\Lambda/\phi_{p^2}(g) \Lambda$.  It is isomorphic as a $\Lambda$-module to the submodule $(g^p-1)\Lambda$ of $\Lambda$, which we also refer to as $S$.
  %In the same spirit as before, consider the quotient ring $\Lambda/\phi_{p^2}(g) \Lambda$. We have $\Lambda/\phi_{p^2}(g) \Lambda \simeq (g^p-1)\Lambda$   and we use the notation $S$ to refer to one of these modules depending on the context; if we want to see $S$ as quotient ring or a submodule of $\Lambda$.
  \item $R$ is the quotient ring $\Lambda/\phi_{p}(g)\Lambda$.  It is isomorphic as a $\Lambda$-module to the submodule $(g-1)\phi_{p^2}(g)\Lambda$ of $\Lambda$, which we also refer to as $R$.
 % For the quotient $\Lambda/\phi_{p}(g)\Lambda$, we have the isomorphism $\Lambda/\phi_{p}(g)\Lambda\iso (g-1)\phi_{p^2}(g)\Lambda$, and we use the notation $R$ to refer to one of these modules depending on the context; if we want treat $R$ as quotient ring or  a submodule of $\Lambda$.  
  \end{itemize}
 There exists a natural isomorphism 
 $$\Lambda=\Z G \iso \Z[x]/(x^{p^2}-1)$$
 which sends $g$ to $x$. This interpretation of $\Lambda$ yields the following interpretations for $E, R$ and $S$ when we treat them as quotient rings of $\Lambda$:
 $$E=\Z[x]/(x^p-1),$$ $$R=\Z[x]/(\phi_p(x))\;\text{and}\; S=\Z[x]/(\phi_{p^2}(x)).$$
 This interpretation yields  obvious isomorphisms of rings: $R\iso \Z[\zeta_{p^2}^p]$ and $S\iso \Z[\zeta_{p^2}]$, where $g$ will be sent to $\zeta_{p^2}^p$ (in $R$) or $\zeta_{p^2}$ (in $S$) via the identification of $g$ with $x$. 
 %where in both cases $g$ will be sent to either $\zeta_{p^2}^p$ or $\zeta_{p^2}$ via the identification of $g$ with $x$. 
 Also, treating  $E$ as a quotient ring of $\Lambda$, the last interpretation of $\Lambda$ gives us an isomorphism between $\overline{E}$ and the ring $\F_p[\lambda]/(\lambda^p)$  where $\lambda:=g-1$. We also have ring isomorphisms $\overline{S}\iso\F_p[\lambda]/(\lambda^{p(p-1)})$ and $\overline{R}=\F_p[\lambda]/(\lambda^{p-1})$. So we get natural projections $$R\longrightarrow \overline{R}\longrightarrow \F_p[\lambda]/(\lambda^m),$$ for $0\leqslant m\leqslant p-1$,  and $$S\longrightarrow\overline{S}\longrightarrow\overline{E}=\F_p[\lambda]/(\lambda^{p}).$$ Denote by $u(A)$ the group of units of a commutative ring $A$. We will refer to "\emph{image of $u(R)$}" as the image of $u(R)$ in $\F_p[\lambda]/(\lambda^m)$ for $0\leqslant m\leqslant p-1$. Similarly, we refer to "\emph{image of $u(S)$}" as the image of $u(S)$ in $\overline{E}$. We define the factor groups as
\begin{equation}\label{Um} \left\{
\begin{array}{cc}

U_m = \displaystyle \frac{ u\left(    \frac{  \F_p[\lambda]  } {\left( \lambda^{m} \right) }  \right) } {\mbox{ image of }  u(R)  }, 1 \leqslant m \leqslant p-1;\\

\\

U_0=\{1\} \mbox{ and } U_p = \displaystyle \frac{ u(\overline{E}) } { \mbox{ image of } u(E)  \cdot \mbox{ image of } u(S) } .\\

\end{array}\right.
\end{equation} For each $1 \leqslant m  \leqslant p,$ let $\widetilde{U}_m$ denote a full set of representatives $u$ (in $u\left(\overline{R}\right)$ or $u\left(\overline{E}\right)$ depending on $m$) of the elements of $U_m,$ where these $u's$ are chosen so that $u \equiv \, 1 (\mbox{mod } \, \lambda )$.

Whenever $A$ is a ring and $X,Y$ are $A$-modules and $\gamma\in \tn{Ext}^1_{A}(Y,X)$, we denote by $(X,Y;\gamma)$ the $A$-module given as the extension
$$0 \longrightarrow X \longrightarrow(X,Y;\gamma) \longrightarrow Y \longrightarrow0$$
with extension class $\gamma$.

%In the classification of indecomposable $\Lambda$-lattices, the elements of the factor groups work as parameters for the extension class of an $S$-lattice by an $R$-lattice (see \cite{reiner1978} for more details). 
The classification of the indecomposable $\Lambda$-lattices is as follows 
(see \cite[Theorem 7.3]{reiner1978}, \cite[Theorem 1]{jones} or \cite[Theorem 34.35]{curtis}).  See for example \cite{wash} for basic definitions regarding ideal class groups.

\begin{theorem} \label{integral indecom} Let $\mathfrak{b}$ range over a full set of representatives of the ideal classes of $R=\Zzeta,$ and $\kc$ likewise for the ideal classes of $S=\Zzetab.$ 
%For each $\kb$ and $\kc,$ let $E(\kb):=(\Z, \kb; 1)$ be the non-split extension of $\kb$ by $\Z$ with extension class $1,$  $E(\kc):=(\Z, \kc; 1)$ be the non-split extension of $\kc$ by $\Z$ with extension class $1,$ and $ (\Z \oplus E(\mathfrak{b}), \mathfrak{c}; 1 \oplus \lambda^{r}u)$ the extension of $\kc$ by $\Z \oplus E(\mathfrak{b})$ with extension class $1 \oplus \lambda^{r}u,$ and so on. 
Let $n_0$ be some  fixed quadratic non-residue ($\mbox{mod } p$), when $p \equiv \, 1 (\mbox{mod } \, 4 ).$ 
Then a full list of non-isomorphic indecomposable $\mathbb{Z}G$-lattices, where $G$ is a cyclic group of order $p^2,$ is as follows \begin{enumerate}
\item [(A)] $\Z, \mathfrak{b},  E(\mathfrak{b})=(\Z, \mathfrak{b}; 1), \mathfrak{c}, E(\mathfrak{c})=(\Z, \mathfrak{c}; 1);$ 

\item[(B)] $ (E(\mathfrak{b}), \mathfrak{c}; \lambda^{r} u), u \in \widetilde{U}_{p-r}, 0 \leqslant r \leqslant p-1;$

\item[(C)] $ (\Z \oplus E(\mathfrak{b}), \mathfrak{c}; 1 \oplus \lambda^{r}u), u \in \widetilde{U}_{p-1-r}, 1 \leqslant r \leqslant p-2;$

\item[(D)] If $p \equiv 1 \, (\mbox{mod } 4),$ $(\Z \oplus E(\mathfrak{b}), \mathfrak{c}; 1 \oplus \lambda^{r} u n_{_{0}}), u \in \widetilde{U}_{p-1-r}, 1 \leqslant r \leqslant p-2;$ 

\item[(E)] $ (\mathfrak{b}, \mathfrak{c}; \lambda^{r} u), u \in \widetilde{U}_{p-1-r}, 0 \leqslant r \leqslant p-2;$

\item[(F)] $ (\Z \oplus \mathfrak{b}, \mathfrak{c}; 1 \oplus \lambda^{r} u), u \in \widetilde{U}_{p-1-r},  0 \leqslant r \leqslant p-2.$    
\end{enumerate}    
\end{theorem}

%$M'\iso \beta(M):=(\widetilde{\beta}(N), \widetilde{\beta}(S'); \widetilde{\beta}(\lambda^ru)).$

Let $\Z_p$ be the ring of $p$-adic integers and, for a finite $p$-group $H$, denote by $M_p$ the pro-$p$ completion of the $\Z H$-lattice $M$. The following result lists all non-isomorphic  indecomposable $\Z_p C_{p^2}$-lattices (see for example Theorem 34.32 in \cite{curtis} or p. 88 in \cite{heller}).

\begin{theorem} \label{p adic indecomposable}
There are precisely $4p+1$ non-isomorphic indecomposable $\mathbb{Z}_p C_{p^2}$-lattices. These are given by \begin{enumerate}
\item [(A)] $\Z_p, R_p=\Zpzeta,  E_p=\Z_pC_p, S_p=\Zpbzetab, (  \Z_p, \Zpbzetab; 1);$ 

\item[(B)] $ (\Z_pC_p, \Zpbzetab; \lambda^{r}), 0 \leqslant r \leqslant p-1;$

\item[(C=D)] $ (\Z_p \oplus \Z_pC_p, \Zpbzetab; 1 \oplus \lambda^{r}), 1 \leqslant r \leqslant p-2;$

\item[(E)] $ (\Zpzeta, \Zpbzetab; \lambda^{r}), 0 \leqslant r \leqslant p-2;$

\item[(F)] $ (\Z_p \oplus \Zpzeta, \Zpbzetab; 1 \oplus \lambda^{r}), 0 \leqslant r \leqslant p-2.$    
\end{enumerate}  
The pro-$p$ completion of a lattice of Type (X) in Theorem \ref{integral indecom} is of Type (X) here.
\end{theorem}

We will say that $M$ and $M'$ are in the same genus if their profinite completions $\widehat{M}$ and $\widehat{M'}$ are isomorphic as $\widehat{\Z} H$-modules, in which case we use the notation $M\vee M'$.  
%As $\widehat{M} = \displaystyle\prod_{q\hbox{ \small{prime}}}M_q$, $M$ and $M'$ are thus in the same genus if and only if $M_q \cong M'_q$ for all primes $q$.  
By \cite[Proposition 31.2]{curtis}, $M\vee M'$ if and only if $M_p \cong M'_p$. 

% By $\Gamma_{_{M}}$ we denote  the set of all $\Lambda$-lattices that are in the same genus as $M.$ According to the Lemma \ref{basta p},  we only need check that  $M_p \cong N_p$ as $\Z_p G$-modules for all primes $p$ dividing $|G|$:   

% \begin{lemma}  [\cite{curtis}, Proposition 31.2] \label{basta p} Let $G$ be a finite group, and let $M$ and $N$ be $\Z G$-lattices. Then $M_q \cong N_q$ as $\Z_q G$-modules for all primes $q$ if and only if $M_p \cong N_p$ as $\Z_p G$-modules for all primes $p$ dividing $|G|.$     
% \end{lemma}

%By the definition of $\widetilde{U}_m$ we can consider the $u's$ and $un_0's$ that appear in the possible direct summands of $M$ as elements of $u(\overline{E}).$ 
Recall that the decomposition of a $\Z G$-lattice as a direct sum of indecomposables is not unique.  Given a direct sum $\bigoplus M_i = M$ of indecomposable $\Lambda$-lattices, define $u_0(M)$ to be the product, in $u(\overline{E})$, of all the $u's$ and $un_0's$ that appear in the given indecomposable summands of $M$ of types $(B)-(F)$, and if there are no such summands, set $u_0(M)=1$. 
%-- the next theorem will tell us that $u_0(M)$ depends only on $M$.  
Define $r_1(M)$ to be the largest exponent $r$ that occurs in a summand of the given decomposition of $M$ of type (B), and if there are no such summands, set $r_1(M)=0$.  Define $r_2(M)$ to be the largest exponent $r$ that occurs in a summand of the given decomposition of $M$ of types $(C)$ to $(F),$ and if there are no such summands, set $r_2(M)=0$.  Finally, set $t:= p-1- \mbox{ max } \{r_2(M),\,\, r_1(M) - 1\}$, unless $M = \Z^{a} \oplus  \displaystyle \bigoplus_{i=1}^{s}\, (E(\mathfrak{b}), \mathfrak{c}; u_i) $ with $s\geqslant 1$, in which case $t:= p$.

%and \begin{equation} \label{rmax} t:= p-1- \mbox{ max } \{r_2(M),\,\, r_1(M) - 1\},\end{equation} to less than $t:=p$ when $M = \Z^{a} \oplus \left( \displaystyle \bigoplus_{i=1}^{s}\, (E(\mathfrak{b}), \mathfrak{c}; u_i) \right)$ and $s \geqslant 1.$

By the $R$-ideal class of a sum of indecomposable $\Lambda$-lattices, we mean the element of the class group of $R$ given as the product of the ideals $\mathfrak{b}$ appearing in the indecomposable summands, and similarly the $S$-ideal class is the product of the $\mathfrak{c}$.  For example, if 
$$M = \Z \oplus \mathfrak{b}\oplus E(\mathfrak{c})\oplus (\Z \oplus E(\mathfrak{b}'), \mathfrak{c}'; 1 \oplus \lambda^{r}u), $$
then the $R$-ideal class is $\mathfrak{b}\mathfrak{b}'$ and the $S$-ideal class is $\mathfrak{c}\mathfrak{c}'$.  The theorem below gives the isomorphism invariants for direct sums of indecomposable $\Z C_{p^2}$-lattices, where notation is as defined after Theorem \ref{p adic indecomposable}.

\begin{theorem} [Theorem 7.8 in \cite{reiner1978} or Theorem 2 in \cite{jones}] \label{invariant sums} Let $M$ be a $\Z G$-lattice, where $G$ is a cyclic group of order $p^2.$ %, which we may assume are of the types listed in Theorem \ref{integral indecom}. 
A full set of isomorphism invariants of $M$ consists of 
\begin{enumerate}
\item [(i)] The genus of $M;$

\item [(ii)] The $R$ and $S$-ideal classes associated with $M;$

\item [(iii)] If $M$ has no summand of types $\kb, E(\kb), \kc, E(\kc)$, then the image of $u_0(M)$ in $U_t;$  

\item [(iv)] If $p \equiv \, 1 (\mbox{mod } \, 4 )$ and $M$ has no summand of type $\Z, E(\kb), E(\kc), (E(\kb), \kc; \lambda^r u)$ or $( \Z \oplus \mathfrak{b}, \mathfrak{c}; 1 \oplus \lambda^{r} u),$ 
then the quadratic character of the image of $u_0(M)$ in $u(\F_p).$ %then the quadratic character $\chi(u_0(M))$ of the image of $u_0(M)$ in $u(\F_p)$.
% ,$ i.e.\ \begin{equation} \nonumber
% \chi(u_0(M))=\left(
% \begin{array}{cc}
% u_0(M)\\ 
% p\\
% \end{array}\right),
% \end{equation} is the symbol of Legendre 
%(see p. 35 in \cite{hua}).     
\end{enumerate}
\end{theorem}

%From now on, let $M$ be a direct sum of indecomposable $\Z G$-lattices listed from $(A)-(F)$  of Theorem \ref{integral indecom}, i.e. 

Given a $\Lambda$-lattice $M$, we may express it as
\begin{equation} \label{forma M} M= \Z^{a} \oplus \displaystyle \bigoplus_{i=1}^{b} \, \mathfrak{b}_i \oplus  \bigoplus_{j=1}^{c} \, \mathfrak{c}_j  \oplus \bigoplus_{k=1}^{d} E(\mathfrak{b}'_k) \oplus \bigoplus_{l=1}^{e} \, E(\mathfrak{c}'_l)  \oplus M_B \oplus M_{C} \oplus M_D \oplus M_E \oplus M_F,\end{equation} where the direct summands $M_B$ -- $M_F$ are of the types $(B)$--$(F),$ respectively. More explicitly, 
\begin{equation} \label{forma MB} 
  M_B = \bigoplus_{r=0}^{p-1} \left( \displaystyle \bigoplus_{k_r= 1 }^{\beta_r}\, \left(E(\mathfrak{b}_{(B,r,k_r)}), \mathfrak{c}_{(B,r,k_r)}; \lambda^r u_{(B,r,k_r)}\right) \right);    
\end{equation}

\begin{equation} \label{forma MCD} 
  M_C =\bigoplus_{r=1}^{p-2} \left( \displaystyle \bigoplus_{l_r= 1}^{\gamma_r}\, \left(\Z \oplus E(\mathfrak{b}_{(C,r,l_r)}), \mathfrak{c}_{(C,r,l_r)}; 1 \oplus \lambda^r u_{(C,r,l_r)}\right) \right);    
\end{equation}

\begin{equation} \label{forma MD} 
  M_D =\bigoplus_{r=1}^{p-2} \left( \displaystyle \bigoplus_{m_r= 1}^{ \delta_r}\, \left(\Z \oplus E(\mathfrak{b}_{(D,r,m_r)}), \mathfrak{c}_{(D,r,m_r)}; 1 \oplus n_0 \cdot \lambda^r \cdot u_{(D,r,m_r)}\right) \right); %(\mbox{ exist if $p \equiv \, 1 (\mbox{mod } \, 4 )$ });    
\end{equation}

\begin{equation} \label{forma ME} 
  M_E = \bigoplus_{r=0}^{p-2} \left( \displaystyle \bigoplus_{s_r= 1}^{\epsilon_r}\, \left(\mathfrak{b}_{(E,r,s_r)}, \mathfrak{c}_{(E,r,s_r)}; \lambda^r u_{(E,r,s_r)}\right) \right);    
\end{equation}

\begin{equation} \label{forma MF} 
  M_F = \bigoplus_{r=0}^{p-2} \left( \displaystyle \bigoplus_{t_r= 1 }^{\eta_r}\, \left(\Z \oplus \mathfrak{b}_{(F,r,t_r)}, \mathfrak{c}_{(F,r,t_r)}; 1 \oplus \lambda^r u_{(F,r,t_r)}\right) \right). 
\end{equation}  
We will sometimes denote $M_{C} \oplus M_D$ as $M_{C,D}.$ Define the tuples: $$\beta:=(\beta_0,...,\beta_{p-1});\; \gamma:=(\gamma_1,...,\gamma_{p-2});\; 
\delta:=(\delta_1,...,\delta_{p-2});$$
$$\epsilon:=(\epsilon_0,...,\epsilon_{p-2})\, \mbox{ and } \, \eta:=(\eta_0,...,\eta_{p-2}).$$
By the previous three theorems, the vector
$$(a,b,c,d,e; \beta, \gamma, \delta, \epsilon, \eta)$$
determines the genus of $M$.

%$\Z G$-lattice $M$ of equation (\ref{forma M}) can thus be briefly denoted as $$M=M(a,b,c,d,e; \beta, \gamma, \delta, \epsilon, \eta).$$ 
%This displays $M$ in terms of the $4p+1$ parameters which are the $p$-genus isomorphism invariants of $M$ (see Theorem \ref{p adic indecomposable}).  

%\begin{remark} \label{are conditions}
%Commented out: put it where it's needed
%In the previous theorem, if $M$ satisfies the following conditions: a) it has summands of the types $(A)$ that are not lattices of fixed points and, b) either $p \not\equiv \, 1 (\mbox{mod } \, 4 )$ or $p \equiv \, 1 (\mbox{mod } \, 4 )$ and between the summands that form $ M$ there is at least one of the types $\Z,$ $E(\kb),$ $E(\kc),$ $(E(\kb), \kc; \lambda^{r} u)$ or $(\Z \oplus \kb, \kc; 1 \oplus \lambda^{r} u)$ then the only isomorphism invariants are given by conditions i) and ii) of the previous result. If $M$ does not have summands of the types $\kb,$ $\kc,$ $E(\kb)$ and $E(\kc)$ and condition b) is  satisfied, the isomorphism invariants are of i) - iii). If $p \equiv \, 1 (\mbox{mod } \, 4 )$  and $M$ do not contain direct summands of the types $\Z,$ $E(\kb),$ $E(\kc),$ $(E(\kb), \kc; \lambda^{ r} u)$ and $(\Z \oplus \kb, \kc; 1 \oplus \lambda^{r} u),$ but contains a summand  of type $\kb$ or $\kc$ so the invariants are i), ii) and iv). Finally if $M$ is a direct sum such that conditions a) and b) are not satisfied then the isomorphism invariants are i)-iv). 
%\end{remark}

\begin{lemma} \label{posto e completamentos}
Let $G$ be cyclic of order $p^2$ and $M$ and $M'$ two $\Z G$-lattices with parameters $(a,b,c,d,e; \beta, \gamma, \delta, \epsilon, \eta)$ and $(a',b',c',d',e'; \beta', \gamma', \delta', \epsilon', \eta')$ respectively.  Then $M$ and $M'$ are in the same genus if, and only if,  
%Let $M=M(a,b,c,d,e; \beta, \gamma, \delta, \epsilon, \eta)$ and $M'=M(a',b',c',d',e'; \beta', \gamma', \delta', \epsilon', \eta')$ be $\Z G$-lattices. Then 
%$\widehat{M} \cong \widehat{M'}$ 
%$M_p\iso M'_p$ 
%as $\Z_pG$-modules if, and only if 
$a=a', b=b', c=c', d=d', e=e', \beta=\beta', \gamma + \delta=\gamma'+ \delta', \epsilon=\epsilon'$ and $\eta=\eta'$.% for each respective $r$ occurring in the sums (\ref{forma MB})-(\ref{forma MF}).

%(\ref{forma MB}), (\ref{forma MCD}), (\ref{forma MD}), (\ref{forma ME}) and (\ref{forma MF}).
\end{lemma}

\begin{proof}
%Firstly, $M\vee M'$ if and only if $M_p \cong M'_p$ as $\Z_p G$-modules. 
The forward implication follows from the previous observations 
%, the final comment of Theorem \ref{p adic indecomposable} 
and the Krull-Schmidt-Azumaya Theorem for $\Z_pG$-modules (see Theorem 6.12 in \cite{curtis}). The backwards implication follows from the final comment in Theorem \ref{p adic indecomposable} and the Krull-Schmidt-Azumaya Theorem.

\end{proof}

\section{Actions of the Galois Group}\label{s3}

Let $G=\langle g\rangle$ be a cyclic group of order $p^2.$  Let $M$ and $M'$ be two $\Z G$-lattices and $\psi_1$, $\psi_2$ the associated representations of $G$. In this section, we deduce a fundamental result to decide when the semidirect products $M\rtimes_{\psi_1}G$ and $M'\rtimes_{\psi_2}G$ are isomorphic as groups. 
%This result will be fundamental to calculate the profinite genus of some semidirect products. 
To obtain the result, we give an action of the Galois group $G(p^2)$ of the cyclotomic extension  $\Q(\zeta_{p^2}) : \Q$ on the genus of any given isomorphism class of indecomposable $\Z G$-lattice.  We maintain the notation from the previous section.

  %Keeping the notation of the previous section, we denote by $\Lambda$ the free $\Z C_{p^2}$-module $\Z G$. By $E$ we denote the submodule of the $\langle g^p\rangle$-fixed points   of $\Lambda,$ in other words, $E=\phi_{p^2}(g) \Lambda\iso \Lambda/(g^p-1)\Lambda.$ By $S$ we denote the quotient $\Z G$-module $\Lambda/E\iso(g^p-1)\Lambda\iso\Z [\zeta_{p^2}]$  and by $R$ the $\Z G$-module $(g-1)E\iso \Lambda/\phi_{p}(g)\Lambda\iso \Z[\zeta_{p^2}^{p}]$. Throughout this section we will  need  two interpretation for $E, R$ and $S$ given before: both as quotient ring of $\Lambda$ and as submodules of $\Lambda$.   

 Note that $G(p^2)\iso \text{Aut}(G)$.  Furthermore, each   $\beta \in \text{Aut}(G)$ induces by the universal property of the algebra $\Z G,$ an isomorphism of algebras  $\widetilde{\beta}:\Z G \longrightarrow \Z G$ that extends $\beta$ uniquely. As mentioned in Section \ref{2}, the  $\Z G$-lattices $E, R$ and $S$ can be treated as submodules of $\Lambda$, and $\widetilde{\beta}$ preserves these submodules.  By Roiter's Lemma \cite[Lemma 31.6]{curtis}, if $L$ is a $\Z G$-lattice in the genus of $X$, where $X$ is $E,R$ or $S$, then there exists an injective map from $L$ to $X$.  Thus we can treat $L$ as a submodule of $\Lambda$. It follows that $\beta$ induces a $\Z$-isomorphism $\widetilde{\beta}|_L:L\longrightarrow \widetilde{\beta}(L)$ that satisfies $\widetilde{\beta}|_L(gl)=\beta(g) \cdot \widetilde{\beta}|_L(l)$, for $g \in G, l \in L$. The isomorphism class of $\widetilde{\beta}(L)$ is in the genus of $L$ and does not depend on the choice of submodule of $\Lambda$. Hence $L\mapsto \widetilde{\beta}(L)$ defines an action of $G(p^2)$ on the genus of $L$.

 For the remaining indecomposable $\Z G$-lattices we begin by defining an action of $G(p^2)$ on the factor groups (\ref{Um}).
%   \begin{equation} \left\{
% \begin{array}{cc}

% U_m = \displaystyle \frac{ u\left(    \frac{  \F_p[\lambda]  } { \left( \lambda^{m} \right) }  \right) } {\mbox{ image of }  u(R)  }, 1 \leqslant m \leqslant p-1;\\

% \\

% U_0=\{1\} \mbox{ and } U_p = \displaystyle \frac{ u(\overline{E}) } { \mbox{ image of } u(E)  \cdot \mbox{ image of } u(S) } .\\

% \end{array}\right.
% \end{equation} 

%Now, let $M$ be an indecomposable $\Z G$-lattice which is a non-split extension $(N,S'; \alpha)$ of the types mentioned in Theorem \ref{integral indecom}, where $\alpha \in \text{Ext}(S, N) \cong \overline{N} \cong \text{Ext}(S', N)$ (see Proposition 4.1 in \cite{reiner1978}). We will show that if $M'$ is a $\Z G$-lattice with a $\Z$-isomorphism $\theta: M\longrightarrow M'$ such that $\theta(gm)=\beta(g) \cdot \theta(m), \, \forall g \in G, m \in M,$ then $M'\iso M^{\beta}:=(\widetilde{\beta}(N), \widetilde{\beta}(S'); \overline{\beta}(\alpha)),$ 

 %, i.e. $\overline{\beta}: \F_p G  \longrightarrow \F_p G$ satisfying \begin{equation} \label{beta overline} \overline{\beta} \left(\displaystyle\sum_{h \in g, f_h \in \F_p} f_h \cdot h\right)= \displaystyle\sum_{h \in g, f_h \in \F_p} f_h \cdot \beta(h),\end{equation} and furthermore, $\overline{\beta}(\overline{N})=\overline{N}.$  

 If $\beta \in G(p^2)$ then we denote by $\overline{\beta}$ the isomorphism of $\F_pG$ induced by $\widetilde{\beta}.$
 %i.e. $\overline{\beta}$ is the $\F_p G$-homomorphism  $\overline{\beta}: \F_p G  \longrightarrow \F_p G$ satisfying \begin{equation} \label{beta overline} \overline{\beta} \left(\displaystyle\sum_{h \in g, f_h \in \F_p} f_h \cdot h\right)= \displaystyle\sum_{h \in g, f_h \in \F_p} f_h \cdot \beta(h).\end{equation} 
 There is a natural epimorphism from $\overline{E} \cong \F_p[\lambda] / (\lambda^p)$ to $\F_p[\lambda] / (\lambda^{m}),$ for each $m\in \{1,2,\ldots, p\}$. %For each $u \in u(\overline{E})$ denote by $\overline{u}$ the image of $u$ in $\left( \F_p[\lambda] / (\lambda^{t}) \right)^{\cdot}$ by this epimorphism. 
 By the definition of $\overline{\beta}$ we have  $\overline{\beta}(\overline{E})=\overline{E}$ and $\overline{\beta}(\F_p[\lambda] / (\lambda^{m}))=\F_p[\lambda] / (\lambda^{m})$ for each $m$.  Furthermore, $G(p^2)$ acts on $\overline{E}$ as $\beta \cdot u = \overline{\beta}(u)$. %, \beta \in G(p^2), u \in u(\overline{E}).$ 
 Denote by $\overline{\beta}_m$ the automorphism induced by $\overline{\beta}$ on $\F_p[\lambda] / (\lambda^{m})$ for each $m\in \{1,2,\ldots, p\}$.  Fix $u\in \widetilde{U}_m$ (as defined after (\ref{Um})) and let $\overline{u} \in U_m$ be its image. %such that $u \in u(\overline{E})$ is one of its respective class representatives according to the definition before the Theorem \ref{integral indecom}. 
 By the definition of $\overline{\beta}$, $G(p^2)$ acts on $U_m$ as $\beta \cdot \overline{u} = \overline{\beta}_m(\overline{u})$, because $\widetilde{\beta}$ preserves $R, S$ and $E$. %such that $\overline{\beta}(u)$ is the class representative to the right of the equality and whose image in $\left( \F_p[\lambda] / (\lambda^{m}) \right)^{\cdot}$ is given by $\overline{\beta_m}(\overline{u}).$ 
 This action on $U_m$ thus induces an action on $\widetilde{U}_m$ for each $m$.
 
%  \begin{remark} \label{imagem do Delta}
%   Note that the natural ring epimorphisms $ \Lambda \longrightarrow E=\Z H$ and $E \longrightarrow R=\Z[\theta_p]$ which apply $g \mapsto h$ and $h \mapsto \theta_p$ respectively, where $\theta_p:=\zeta_{_{p^2}}^p$ and $H=\left \langle h \right \rangle \cong C_p,$ show that the action of $G=\left \langle g\right \rangle$ on the elements of $E$ and $R$ are given by $g \cdot e = h e$ and $g \cdot r= r\theta_p$, respectively.     
% \end{remark}

\begin{lemma} \label{3} Let $\beta \in G(p^2)$ and let $u \in \widetilde{U}_m$. For each  $r \in \{0, 1, \ldots , p-1\}$ we have 
$$\overline{\beta}(\lambda^r \cdot u)= \lambda^r \cdot \overline{\beta}( u).$$ % where $\overline{\beta}(u)  \in u(\overline{R})$ or $\overline{\beta}(u)  \in u(\overline{E}).$  % depending on $u \in u(\overline{R})$ or $u(\overline{E}).$  
\end{lemma}

\begin{proof} The result is trivial if $r=0$ because $\lambda^0 =1.$ Assume from now on that $r \geqslant 1.$ By the definition of $\widetilde{U}_m$ we have that $u \in u(\overline{R})$ or $u \in u(\overline{E}),$ and since $\overline{\beta}(\overline{R})=\overline{R}$ and $\overline{\beta}(\overline{E})=\overline{E},$ we have that  $\overline{\beta}(u) \in u(\overline{R})$ or $\overline{\beta}(u) \in u(\overline{E}).$ Suppose $\beta(g)=g^l,$ where $l\leqslant p^2$ is an integer coprime to $p$.  Set $\Delta_{l}(g)=1+ g + \cdots + g^{l-1}.$ 
So we have \begin{equation} \label{conta dos lambas r} \overline{\beta}(\lambda^r \cdot u)= \left(\overline{\beta}({\lambda})\right)^{r} \overline{\beta}(u) =  ( \beta(g) -1 )^{r} \overline{\beta}(u) = (g^l -1)^r \overline{\beta}(u) = (g-1)^r  (1+ g + \cdots + g^{l-1})^r \overline{\beta}(u)\end{equation} whence it follows that $\overline{\beta}(\lambda^r \cdot u)=  \lambda^r \cdot (\Delta_{l}(g))^r \cdot \overline{\beta}(u).$ 
%The image of $\Delta_l(g)$ in $R$ is given by  $\Delta_{l'}(g)=1+ g + \cdots + \theta^{l'-1}$ where $1 \leqslant l'  \leqslant p-1$ and $l' \cong l \, (\mbox{mod } p)$ because $\{1, g, \ldots, g^{p-2}\}$ is a $\Z$-base of $R$ and $ g^{p-1} = -1- g- \cdots- g^{p-2}.$ 
By \cite[Lemma 1.3]{wash}, $\Delta_l(g) = \frac{g^{l} -1}{g -1}$ 
is a unit of $R$.  Therefore $(\Delta_{l}(g))^r \cdot \overline{\beta}(u)$ and $\overline{\beta}(u)$ have the same image in $U_m$, because $U_m$ is the quotient of  the units of $\frac{  \F_p[\lambda]  }{ \left( \lambda^{m} \right)}$ by the units of $R$. % $  = \displaystyle \frac{ \left(    \frac{  \F_p[\lambda]  } { \left( \lambda^{m} \right) }  \right)^{\cdot} } {\mbox{ image of }  u(R)},$ 
So $\overline{\beta}(\lambda^r \cdot u)= \lambda^r \cdot \overline{\beta}(u).$ 
\end{proof}

Given an indecomposable $\Z G$-lattice $M$ and $\beta\in G(p^2)$, we define the indecomposable $\Z G$-lattice $M^\beta$ as follows: for Type (A),
$$\Z^{\beta} = \Z,\, \mathfrak{b}^\beta = \beta(\mathfrak{b}),\, E(\mathfrak{b})^{\beta} = E(\beta(\mathfrak{b})),\, \mathfrak{c}^{\beta} = \beta(\mathfrak{c})\, \mbox{ and } E(\mathfrak{c})^{\beta} = E(\beta(\mathfrak{c})).$$
Types (B)--(F) are extensions $(X,Y;\gamma)$, and we define $(X,Y;\gamma)^{\beta} := (X^{\beta}, Y^\beta; \overline\beta(\gamma))$, where $\overline\beta(\gamma)$ is as established in Lemma \ref{3}.  For example, for a module of Type (D), 
$$(\Z \oplus E(\mathfrak{b}), \mathfrak{c}; 1 \oplus \lambda^{r} u n_{_{0}})^{\beta} = 
(\Z \oplus E(\beta(\mathfrak{b})), \beta(\mathfrak{c}); 1 \oplus \lambda^{r} \overline\beta(u )n_{_{0}}).$$
 
  Let $M, N$ and $N'$ be $\Z G$-lattices, and  $\xi \in \text{Ext}^{1}_{\Lambda}(M,N)$. %If a $\Z G$-module homomorphism of modules, say $\gamma: N\longrightarrow N'$, is given,  
  If  a $\Z G$-module homomorphism $\gamma: N\longrightarrow N'$ is given, then we get an extension 
 $\gamma\xi \in \text{Ext}^{1}_{\Lambda}(M,N')$  through the  commutative diagram below: 
 $$\xymatrix{\xi:0\ar[r]& N\ar[r]^{i}\ar[d]^{\gamma}&X\ar[r]^{\pi}\ar[d]&M\ar[r]\ar@{=}[d]&0\\
 \gamma\xi:0\ar[r]& N'\ar[r]^{i'}&\textsubscript{$\gamma$}X\ar[r]^{\pi'}&M\ar[r]&0}$$
 where $\textsubscript{$\gamma$}X$ is the pushout of the maps $i$ and $\gamma$. 
 Now, if $M'$ is another $\Z G$-lattice with a $\Z G$-module homomorphism $\varphi: M'\longrightarrow M$ then we obtain an extension $\gamma\xi\varphi\in \text{Ext}^{1}_{\Lambda}(M',N')$ through the commutative diagram below:
 
 $$\xymatrix{\xi:0\ar[r]& N\ar[r]^{i}\ar[d]^{\gamma}&X\ar[r]^{\pi}\ar[d]&M\ar[r]\ar@{=}[d]&0\\
 \gamma\xi:0\ar[r]& N'\ar[r]^{i'}&\textsubscript{$\gamma$}X\ar[r]^{\pi'}&M\ar[r]&0\\
 \gamma\xi\varphi:0\ar[r]& N'\ar[r]^{i''}\ar@{=}[u]&\textsubscript{$\gamma$}X_{\varphi}\ar[r]^{\pi''}\ar[u]&M'\ar[r]\ar[u]^{\varphi}&0}$$ where $\textsubscript{$\gamma$} X_\varphi$ is the pullback of the maps $\pi'$ and $\varphi$.
 When $M\vee M'$,  $N\vee N'$ and $\gamma$ and $\varphi$ induce $\Z_p G$-isomorphisms $\gamma_p: N_p\longrightarrow N'_p$ and $\varphi_p:M'_p\longrightarrow M_p$, then the map $t:\text{Ext}^{1}_{\Lambda}(M,N)\longrightarrow\text{Ext}^{1}_{\Lambda}(M',N')$ defined by $t(\xi)=\gamma\xi\varphi$ is an isomorphism of abelian groups, called the Standard Isomorphism associated with the pair $(\gamma, \varphi)$ (see \cite[p.\ 716-717]{curtis}). The following result will be used in the remainder of this text. 

\begin{lemma}(\cite[Theorem 34.14]{curtis}) \label{padrao}
    Let $L,L',N,N'$ be $\Lambda$-lattices with  $L \vee L'$ and $N \vee N'$, and let  $t: \text{Ext}^{1}_{\Lambda}(N, L) \longrightarrow \text{Ext}^{1}_{\Lambda}(N', L')$ be a standard isomorphism. Then $t$ induces a bijection between the set of isomorphism classes of extensions of $N$ by $L,$ and the corresponding set for $N', L'$.    
  % Let $L$ and $N$ be left $\Lambda$-lattices, and let $L \vee L',$ $N \vee N'$ and $t: Ext^{1}_{\Lambda}(N, L) \longrightarrow Ext^{1}_{\Lambda}(N', L')$ be a standard isomorphism. Then $t$ induces a bijection between the set of isomorphism classes of extensions of $N$ by $L,$ and the corresponding set for $N', L'$.    
\end{lemma}

\begin{remark}
    The cited result \cite[Theorem 34.14]{curtis} demands that the lattices involved in the lemma above be ``Eichler lattices'': such a lattice $M$ is defined by prohibiting certain non-commutative indecomposable summands in the endomorphism ring of $\Q\otimes_{\Z}M$.  But for abelian groups, such summands cannot occur, so every $\Z G$-lattice is an Eichler lattice.
\end{remark}

This result shows that the classification of the  indecomposable $\Z G$-lattices that are described in Theorem \ref{integral indecom} depends only on their genus. We will use this fact to describe the action of $G(p^2)$ on the indecomposable $\Z G$-lattices. Furthermore, via a standard isomorphism, each indecomposable $\Z G$-lattice of the form $(X',Y'; \xi')$ is the image of an indecomposable $\Z G$-lattice which has either  the form  $(\Z,Y;1)$ or $(X,S;\xi)$, where $X$ is one of the modules  $E$, $R$, $\Z\oplus E$ or $\Z\oplus R$ and $Y$ is $S$ or $R$. So for calculation  purposes, we shall fix a set of homomorphisms that induce the standard isomorphisms which are utilized to describe the indecomposable $\Z G$-lattices of Theorem \ref{integral indecom}. This fixed set will be used implicitly in the remainder of this text. Let $S, S_1, S_2, \ldots S_n$ be representatives of the orbits of the action of $G(p^2)$ on the genus of $S$  --- or in other words, the $G(p^2)$-orbits of the class group $H(\Z[\zeta_{p^2}])$. Of course, every $\Z G$-lattice in the genus of $S$ is of the form $S_i^{\beta}$ for some $i$ and  $\beta \in G(p^2)$.  
 Let $\varphi_{i,S}: S_i \longrightarrow S$ be a $\Z G$-module homomorphism such that $(\varphi_{i,S})_p$ is an isomorphism for each $i$, so that in particular $\varphi_{i,S}$ is injective. Let $\beta\in G(p^2)$, and define $\varphi_{i,S}^{\beta}: S_i^{\beta} \longrightarrow S$ as the $\Z G$-module homomorphism given by $\varphi_{i,S}^{\beta} = \widetilde{\beta} \varphi_{i,S} \widetilde{\beta}^{-1}.$ Likewise, let $R, R_1, R_2, \ldots R_m$ be representatives of the orbits of the action of $G(p^2)$ on the genus of $R$ and let $\varphi_{j,R}:R_j \longrightarrow R$ be a $\Z G$-module homomorphism such that $(\varphi_{j,R})_p$ is an isomorphism for each $j$, and define $\varphi_{j,R}^{\beta}: R_j^{\beta} \longrightarrow R$ as $\varphi_{j,R}^{\beta} = \widetilde{\beta} \varphi_{j,R} \widetilde{\beta}^{-1}$.
 Let %$\varphi_{R,j}: R \longrightarrow R_j$ be a $\Z G$-module homomorphism such that $(\varphi_{R,j})_p$ is an isomorphism for each $j$, also let 
 $\varphi_{R,j}: R \longrightarrow R_j, \varphi_{E,j}: E \longrightarrow E(R_j)$,  $\varphi_{\Z\oplus E,j}: \Z\oplus E \longrightarrow \Z\oplus E(R_j)$ and $\varphi_{\Z\oplus R,j}: \Z\oplus R \longrightarrow \Z\oplus R_j$ be $\Z G$-module homomorphisms such that their $p$-adic completions are isomorphisms.  Then,  
  for each $\beta\in G(p^2)$, define $\varphi_{E,j}^{\beta}: E\longrightarrow E(R_j)^{\beta}$ as the $\Z G$-module homomorphism given by $\varphi_{E,j}^{\beta} = \widetilde{\beta} \varphi_{E,j} \widetilde{\beta}^{-1}$ and similarly define $\varphi_{X,j}^{\beta}$ for $X=R,X=\Z\oplus E$ and $X=\Z\oplus R.$ From now on, when we write $(X',S';\xi)$ where $S'=\widetilde{\beta}(S_i)$  for some $\beta \in G(p^2)$ and $X'$ is of the form $(X'')^{\alpha}$ where $X''$ is one of the modules $E(R_j), R_j, \Z\oplus E(R_j)$ or $\Z\oplus R_j$ for some $\alpha \in G(p^2)$, we mean the extension $(X',S';\xi)=\textsubscript{$\varphi^\alpha_{X,j}$} (X,S;\xi)_{\varphi_{i,S}^\beta}$ with $X$ being $E, R, \Z\oplus E$ or $\Z\oplus R$. The extensions of the form $E(Y)=(\Z,Y;1)$  where $Y$ is in the genus of $R$ or $S$ will be treated similarly. 
%\begin{defn} \label{semilinear}
%Let $M$ and $N$ be $\Z G$-modules. A \emph{semi-linear homomorphism} from $M$ to $N$ is a pair $(f, \beta),$ where $f:M \longrightarrow N$ is an abelian group homomorphism and $\beta \in \text{Aut}(G)$ such that $f(x \cdot m)=\beta(x) \cdot f(m)$ for $x \in G$ and $m \in M.$ 
%\end{defn}

%\textcolor{red}{Check this proof and its text.  The words ``semi-linear homomorphism'' should be used where appropriate}

\begin{lemma} \label{caso genus} 
Let $M$ be an indecomposable $\Z G$-lattice. % that is a non-split extension. % (that is, of types (B)--(F) in Theorem \ref{integral indecom}). 
%of the types given in Theorem \ref{integral indecom}. 
Then for each $\beta \in G(p^2)$ there is a $\Z$-isomorphism $\Theta_\beta: M\longrightarrow M^{\beta}$ such that 
$$\Theta_\beta(g' \cdot m)=\beta(g') \cdot \Theta_\beta(m),\quad \forall g' \in G, m \in M.$$
%Let $G=\left \langle g \right \rangle$ be a cyclic group of order $p^2.$ Suppose that $M$ is an indecomposable $\Z G$-lattice which is a non-split extension of the form $M=(\Z, R; 1)$ or $M=(N, S; \alpha)$, where $N$ is one of the following types: $\Z, E, \Z \oplus E, R$ or $\Z \oplus R$. Then for  each $\beta \in G(p^2)$ there is a $\Z$-isomorphism $\Omega_{\beta}: M\longrightarrow M^{\beta}$ such that $\Omega_{\beta}(g' \cdot m)=\beta(g') \cdot \Omega_{\beta}(m), \, \forall g' \in G, m \in M.$
\end{lemma}

\begin{proof}

When $M$ isn't a non-split extension the result is clear.  We work through the case of a module $(\Z\oplus R,S; 1\oplus\lambda^ru)$ of Type (F), which is the most complicated.  The remaining cases are similar.
%The case $(\Z\oplus R,S; 1\oplus\lambda^ru)$ (type (F)). The exact sequence 

We treat $E$ as the submodule $\phi_{p^2}(g)\Lambda$ of $\Lambda$ as at the start of Section   \ref{2}.  The inclusion of $E$ in $\Lambda$ induces the short exact sequence 
$$\xymatrix{0\ar[r]&E\ar[r]^{i_0}& \Lambda\ar[r]^{\pi_{0}}&S\ar[r]&0}.$$ 
Applying the functor $\Hom_{\Lambda}(-,\Z\oplus R)$, we obtain 
$$\text{Ext}_{\Lambda}^1(S,\Z\oplus R) \cong \Hom_{\Lambda}(E,\Z\oplus R)/\tn{Im}(\Hom_{\Lambda}(i_0,\Z\oplus R)) \cong \overline{R}\oplus\F_p \cong \frac{  \F_p[\lambda]  } { \left( \lambda^{p-1} \right) } \oplus \F_p.$$ 
Every map $f\in \Hom_{\Lambda}(E,\Z\oplus R)$ is completely determined by the image of $\phi_{p^2}(g)$. So let $f\in \Hom_{\Lambda}(E,\Z\oplus R)$  be the map such that $f(\phi_{p^2}(g))=1\oplus\lambda^rw$ where $u$ and $w+ pR$ coincide in $\overline{R}$. Treating $1\oplus\lambda^ru$ as an element of $\text{Ext}_{\Lambda}^1(S,\Z\oplus R)$ we have the following commutative diagram
\begin{equation}\label{M push}\xymatrix{& 0\ar[r]&E\ar[r]^{i_0}\ar[d]^f& \Lambda\ar[r]^{\pi_0} \ar[d]&S\ar[r]\ar@{=}[d]&0\\
1\oplus\lambda^ru: & 0\ar[r]&\Z\oplus R\ar[r]^{i}& M\ar[r]^{\pi}&S\ar[r]&0}
\end{equation} 
where $M$ is the pushout of the maps $i_0$ and $f$. Explicitly, $$M=(\Lambda\oplus \Z\oplus R)/\langle(i_0(y),-f(y))| y \in E\rangle.$$ Let $\beta \in G(p^2)$ and consider the map $\widehat{\beta}: \Lambda\oplus \Z\oplus R\longrightarrow \Lambda\oplus \Z\oplus R$ given by $(\lambda, z, r)\mapsto (\widetilde{\beta}(\lambda), z, \widetilde{\beta}(r))$.  Denote by $e$ the element $\phi_{p^2}(g)$. We have $$\widehat{\beta}(i_0(e), -f(e))=(\widetilde{\beta}(e), -\widetilde{\beta}(f(e)))=
(e, -\widetilde{\beta}(f(e)) )=(i_0(e), -\widetilde{\beta}(f(e)) ).$$ So  if  $\widetilde{\beta}(f):E\longrightarrow \Z\oplus R$ is the $\Z G$-module homomorphism which sends $e$ to $\widetilde{\beta}(f(e))$ then $\widehat{\beta}$ induces a $\Z$-isomorphism $$\Omega_{\beta}: M=(\Lambda\oplus \Z\oplus R)/\langle(i_0(y), -f(y))| y \in E)\rangle\longrightarrow (\Lambda\oplus \Z\oplus R)/\langle(i_0(y),-\widetilde{\beta}(f)(y)) | y \in E\rangle,$$ because $\widetilde{\beta}$ is an automorphism of $\Lambda$. 
Since  $$(\Lambda\oplus \Z\oplus R)/\langle(i_0(y),-\widetilde{\beta}(f)(y))\,|\, y \in E\rangle$$ is the pushout of the maps $i_0$ and $\widetilde{\beta(}f)$, it is the module represented by the element $\widetilde{\beta}(f)(e) = \overline{\beta}(1\oplus \lambda^ru)$ of $\text{Ext}_{\Lambda}^1(S, \Z\oplus R)$.  Hence
$$(\Lambda\oplus \Z\oplus R)/\langle(i_0(y),-\widetilde{\beta}(f)(y))|\, y \in E\rangle=(\Z\oplus R,S; \overline{\beta}(1\oplus \lambda^ru))=(\Z\oplus R,S; 1\oplus \overline{\beta}(\lambda^ru)).$$ 
Therefore the $\Z$-isomorphism $\Omega_{\beta}:M\longrightarrow M^{\beta}=(\Z\oplus R,S; 1\oplus \overline{\beta}(\lambda^ru))$ satisfies the equation  
$$\Omega_\beta(g' \cdot m)=\beta(g') \cdot \Omega_\beta(m), \, \forall g' \in G, m \in M,$$ and by Lemma \ref{3}, $M^{\beta}=(\Z \oplus R,S; 1 \oplus \lambda^r \cdot \overline{\beta}(u)),$ with $\overline{\beta}(u) \in u(\overline{R}).$ 

 Now we treat the case of the module $(\Z\oplus R',S'; 1\oplus\lambda^ru)$ such that  $(\Z\oplus R') \vee (\Z\oplus R),$ $S'\vee S$. If $S'=\mu(S_i)$ and $R'=\nu(R_j)$ for some $\mu, \nu \in G(p^2)$, we shall use the fixed $\Z G$-module homomorphisms $$\varphi_{i,S}^\mu:S'\longrightarrow S \text{\;and\;} \varphi_{\Z\oplus R,j}^\nu:\Z\oplus R\longrightarrow \Z\oplus R',$$ which induce isomorphisms when completed. To simplify the notation, we write $\varphi=\varphi_{i,S}^\mu$ and $\gamma=\varphi_{\Z \oplus R, j}^\nu$.   These conventions established, 
   consider the module $M=(\Z\oplus R, S; 1 \oplus \lambda^r u)$ and the module  $M^{\beta}$ constructed as above.
    %$$ = (\Lambda\oplus \Z\oplus R)/\langle(i_0(y),-\gamma(y))| y \in E\rangle=(\Lambda\oplus \Z\oplus R)/ W_f,$ where $W_\gamma=$ 
 %and 
 %$$(M_0)^{\beta}=(\Z\oplus \widetilde{\beta}(R), \widetilde{\beta}(S); 1 \oplus \overline{\beta}(\lambda^r u))=(\Z\oplus R, S; 1 \oplus \overline{\beta}(\lambda^r u) ).$$ 
  We have the following commutative diagram    
 \begin{equation} \label{pullback gamma M}
 \xymatrix{
 1\oplus\lambda^ru:&0\ar[r]& \Z\oplus R\ar[r]^{i}\ar[d]^{\gamma}&M\ar[r]^{\pi}\ar[d]&S\ar[r]\ar@{=}[d]&0\\
 \gamma(1\oplus\lambda^ru):&0\ar[r]& \Z\oplus R'\ar[r]^{i'}&\textsubscript{$\gamma$}M\ar[r]^{\pi'}&S\ar[r]&0\\
 \gamma(1\oplus\lambda^ru)\varphi:&0\ar[r]& \Z\oplus R'\ar[r]^{i''}\ar@{=}[u]&\textsubscript{$\gamma$}(M)_{\varphi}\ar[r]^{\pi''}\ar[u]&S'\ar[r]\ar[u]^{\varphi}&0}\end{equation} 
  where $\textsubscript{$\gamma$}M$ is the pushout of $i$ and $\gamma$, and $\textsubscript{$\gamma$}(M)_{\varphi}$ is the pullback of $\pi'$ and $\varphi$. It follows  that $(\Z\oplus R',S'; 1\oplus\lambda^ru)= \textsubscript{$\gamma$}(M)_{\varphi}$. Define \begin{equation} \label{as linhas} \gamma'=\widetilde{\beta}\gamma\widetilde{\beta}^{-1}:\Z\oplus R\longrightarrow \widetilde{\beta}(\Z\oplus R')=\Z \oplus \widetilde{\beta}(R') \mbox{ and } \varphi'=\widetilde{\beta}\varphi\widetilde{\beta}^{-1}:\widetilde{\beta}(S')\longrightarrow S(=\widetilde{\beta}(S)). \end{equation} 
  Observe that  
  $$\gamma'=\widetilde{\beta}\gamma\widetilde{\beta}^{-1}=\widetilde{\beta}\varphi_{\Z\oplus R,j}^\nu\widetilde{\beta}^{-1}= \varphi_{\Z\oplus R,j}^{\beta\nu}$$ and 
  $$\varphi'=\widetilde{\beta}\varphi\widetilde{\beta}^{-1}=\widetilde{\beta}\varphi_{i,S}^\mu\widetilde{\beta}^{-1}=\varphi_{i,S}^{\beta\mu},$$ so $\gamma'$ and $\varphi'$ are among the maps fixed before.  The maps $\gamma'$ and $\varphi'$ give rise to the commutative diagram 
\begin{equation} \label{pullback gamma betaM} \xymatrix{ &0\ar[r]&E\ar[r]^{i_0}\ar[d]^{\widetilde{\beta}(f)}& \Lambda\ar[r]^{\pi_0} \ar[d]&S\ar[r]\ar@{=}[d]&0\\1\oplus\overline{\beta}(\lambda^ru):&0\ar[r]& \Z\oplus R\ar[r]^{j}\ar[d]^{\gamma'}&M^{\beta}\ar[r]^{\alpha}\ar[d]&S\ar[r]\ar@{=}[d]&0\\
 \gamma'(1\oplus\overline{\beta}(\lambda^ru)):&0\ar[r]& (\Z\oplus R')^\beta\ar[r]^{j'}&\textsubscript{$\gamma'$}M^{\beta}\ar[r]^{\alpha'}&S\ar[r]&0\\
 \gamma'(1\oplus \overline{\beta}(\lambda^ru))\varphi':&0\ar[r]& (\Z\oplus R')^\beta\ar[r]^{j''}\ar@{=}[u]&\textsubscript{$\gamma'$}M^{\beta}_{\varphi'}\ar[r]^{\alpha''}\ar[u]&(S')^\beta\ar[r]\ar[u]^{\varphi'}&0}\end{equation} 
 where $\textsubscript{$\gamma'$}M^{\beta}_{\varphi'}=(\Z\oplus (R')^\beta, (S')^\beta; 1 \oplus \overline{\beta}(\lambda^ru)).$  We have $$\textsubscript{$\gamma$}M=(M\oplus \Z\oplus R')/\langle(i(w'),-\gamma(w'))| w' \in \Z \oplus R\rangle$$ and $$  
  \textsubscript{$\gamma'$} M^{\beta}=(M^{\beta}\oplus \widetilde{\beta}(\Z\oplus R'))/\langle(j(w'),-\gamma'(w')) | w' \in \Z \oplus R\rangle.$$ There is  an obvious $\Z$-isomorphism $$\theta_\beta: M\oplus (\Z\oplus R') \longrightarrow M^{\beta}\oplus (\Z\oplus \widetilde{\beta}(R')),$$  which applies $(m, x)$ to $(\Omega_{\beta} (m), \widetilde{\beta}(x))$ where $\Omega_{\beta}: M \longrightarrow M^{\beta}$ has been  constructed previously and  such that $$\Omega_{\beta}(g' \cdot  m)=\beta(g') \cdot \Omega_{\beta}(m), \, \forall g' \in G, m \in M.$$  Let $w' \in \Z \oplus R$. We get   $$i(w')=(0, w')+ \langle(i_0(y),-f(y))| y \in E\rangle \in M,$$  because $M$ is the pushout of the map $i_0$ and $f$ as in diagram (\ref{M push}). By the construction of $\Omega_{\beta}$ and  by the construction of $j$ in diagram (\ref{pullback gamma betaM}), we have  $$\Omega_{\beta}(i(w'))=(0, \widetilde{\beta}(w'))+ \langle(i_0(y),-\beta(f)(y))| y \in E\rangle=j(\widetilde{\beta}(w'))\in M^\beta.$$ So for each $w' \in \Z \oplus R$  we have $$\theta_\beta(i(w'),-\gamma(w'))=(\Omega_\beta(i(w')),-\widetilde{\beta}(\gamma(w')))=(j(\widetilde{\beta}(w')),-\gamma'(\widetilde{\beta}(w')))$$ 
  where the last equality follows from the  equations (\ref{as linhas}).    
%Therefore, for each $w' \in \Z \oplus R$ we have $$\theta_\beta(i(w'),-\gamma(w'))=(\widetilde{\beta}(i(w')),-\widetilde{\beta}(\gamma(w')))=(\widetilde{\beta}(w'),-\gamma'(\widetilde{\beta}(w')))=(j(\widetilde{\beta}(w')),-\gamma'(\widetilde{\beta}(w'))).$$
  This shows that  $\theta_\beta$ induces a $\Z$-isomor\-phism $\widehat{\theta}_\beta:\textsubscript{$\gamma$}M\longrightarrow \textsubscript{$\gamma'$} M^{\beta}$ that satisfies $\widehat{\theta}_{\beta}(g' \cdot \textsubscript{$\gamma$}m)=\beta(g') \cdot \widehat{\theta}_{\beta}(\textsubscript{$\gamma$}m),$ where $\textsubscript{$\gamma$}m \in \textsubscript{$\gamma$}M$ and $g' \in G.$

  Finally, define $$\Theta_\beta': \textsubscript{$\gamma$}M\oplus S'\longrightarrow \textsubscript{$\gamma'$} M^{\beta} \oplus \widetilde{\beta}(S')$$ by $\Theta_\beta'(\textsubscript{$\gamma$}m,s')=(\widehat{\theta}_\beta(\textsubscript{$\gamma$}m), \widetilde{\beta}(s')),$ where $\textsubscript{$\gamma$}m \in \textsubscript{$\gamma$}M$ and $s' \in S'.$ Then $\Theta_\beta' $ is a $\Z$-isomorphism and satisfies  $$\Theta_\beta'(g' \cdot \textsubscript{$\gamma$}m, g' \cdot s')=\beta(g') \cdot \Theta_\beta'(\textsubscript{$\gamma$}m, s'),\, \forall g' \in G.$$
  We must verify that $\Theta_\beta'$ induces a $\Z$-isomorphism $$\Theta_\beta: \textsubscript{$\gamma$}(M)_\varphi\longrightarrow \textsubscript{$\gamma'$} M^{\beta}_{\varphi'}.$$
  Let $(\textsubscript{$\gamma$}m,s')\in \textsubscript{$\gamma$
  }(M)_\varphi$, so if $$\textsubscript{$\gamma$}m=(m',e')+\langle(i(z),-\gamma(z))| z \in \Z \oplus R\rangle,$$ with $(m', e') \in M \oplus (\Z \oplus R'),$ then $\varphi(s')=\pi'(\textsubscript{$\gamma$}m)=\pi(m') \in S$, because $\textsubscript{$\gamma$}M_{\varphi}$ is the pullback of $\pi'$ and $\varphi$ (see diagram (\ref{pullback gamma betaM})). By the definition of $\widehat{\theta}_\beta$, we have $$\widehat{\theta}_\beta(\textsubscript{$\gamma$}m)=\widehat{\theta}_\beta((m',e')+\langle(i(z),-\gamma(z))| z \in \Z \oplus R\rangle)=(\Omega_{\beta}(m'), \widetilde{\beta}(e')) + \langle(j(z),-\gamma'(z))| z \in \Z \oplus R\rangle.$$  It follows from equations (\ref{as linhas})  that  $$\varphi'(\widetilde{\beta}(s'))=\widetilde{\beta}\varphi(s')=\widetilde{\beta}(\pi(m')).$$
  If $m'=(x,w)+ \langle(i_0(y),-f(y))| y \in E\rangle \in M,$ where $ x\in \Lambda$ and $w \in \Z \oplus R$, then $\pi(m')=\pi_0(x) \in S$. So $\widetilde{\beta}(\pi(m'))=\widetilde{\beta}(\pi_0(x))=\pi_0(\widetilde{\beta}(x))$. On the other hand, $$\Omega_{\beta}(m')=(\widetilde{\beta}(x), \widetilde{\beta}(w)) + \langle(i_0(y),-\widetilde{\beta}(f)(y))| y \in E\rangle \in M^{\beta},$$ hence, it follows from diagram (\ref{pullback gamma betaM}) that $$\alpha(\Omega_{\beta}(m'))=\pi_0(\widetilde{\beta}(x)).$$ As  
  $$ 
  \alpha'(\widehat{\theta}_{\beta}(\textsubscript{$\gamma$}m))=\alpha(\Omega_{\beta}(m')),$$ we have $$\varphi'(\widetilde{\beta}(s'))=\alpha'(\widehat{\theta}_{\beta}(\textsubscript{$\gamma$}m)),$$ and the pair $(\widehat{\theta}_{\beta}(\textsubscript{$\gamma$}m), \widetilde{\beta}(s')) \in \textsubscript{$\gamma'$} M^{\beta}_{\varphi'}.$
   Thus $\Theta_\beta'$ induces a map ${\Theta}_\beta: \textsubscript{$\gamma$}(M)_\varphi\longrightarrow \textsubscript{$\gamma'$} M^{\beta}_{\varphi'}$ which is an injective $\Z$-homomorphism. On the other hand, if $(\tilde{m}, \tilde{s}) \in \textsubscript{$\gamma'$} M^{\beta}_{\varphi'},$ it follows from the fact that $\Theta_{\beta}'$ is a $\Z$-isomorphism that there exists $(\textsubscript{$\gamma$}m, s') \in \textsubscript{$\gamma$}M \oplus S'$ such that $\Theta_{\beta}'(\textsubscript{$\gamma$}m, s')=(\tilde{m}, \tilde{s})=(\widehat{\theta}_{\beta}(\textsubscript{$\gamma$}m), \widetilde{\beta}(s')),$ and therefore $\varphi'(\widetilde{\beta}(s'))=\alpha'(\widehat{\theta}_{\beta}(\textsubscript{$\gamma$}m)).$ As before, write $$\widehat{\theta}_{\beta}(\textsubscript{$\gamma$}m)=(\Omega_{\beta}(m'), \widetilde{\beta}(e'))+ \langle(j(z),-\gamma'(z))| z \in \Z \oplus R\rangle),$$
    where $(\Omega_{\beta}(m'), \widetilde{\beta}(e')) \in M^{\beta} \oplus (\Z \oplus \widetilde{\beta}(R'))$.  It follows from equations (\ref{as linhas}) and from diagrams (\ref{pullback gamma M}) and (\ref{pullback gamma betaM}) that $$\alpha'(\widehat{\theta}_{\beta}(\textsubscript{$\gamma$}m))=\varphi'(\widetilde{\beta}(s'))=\widetilde{\beta}(\varphi(s')) \mbox{ and } \alpha'(\widehat{\theta}_{\beta}(\textsubscript{$\gamma$}m))=\alpha(\Omega_{\beta}(m'))=\widetilde{\beta}(\pi(m'))).$$   Since $\widetilde{\beta} $ is a $\Z$-isomorphism, we have $\varphi(s')=\pi(m')=\pi'(\textsubscript{$\gamma$}m)$, thus $(\textsubscript{$\gamma$}m, s') \in \textsubscript{$\gamma$}(M)_\varphi$ and
     $$\Theta_\beta: \textsubscript{$\gamma$}(M)_\varphi\longrightarrow \textsubscript{$\gamma'$}(M^{\beta})_{\varphi'}$$ is surjective, therefore a $\Z$-isomorphism, which we have already checked has the desired property.
  \end{proof}
  
\begin{defn}Suppose that $M=\displaystyle\bigoplus_{i=1}^k M_i$ where each $M_i$ is an indecomposable $\Z G$-lattice. If  $\beta\in G(p^2)$, then we define $M^{\beta}:=\displaystyle\bigoplus_{i=1}^k \,(M_i)^{\beta}$.
\end{defn}
Let $M=\displaystyle\bigoplus_{i=1}^k M'_i$ be another decomposition of $M$ as a sum of indecomposable $\Z G$-lattices. If $\mathfrak{b}$ and $\mathfrak{c}$ are the $R$ and $S$-ideal classes associated to the decomposition $M=\displaystyle\bigoplus_{i=1}^k M_i$ and $\mathfrak{b'}$ and $\mathfrak{c'}$ are the $R$ and $S$-ideal classes associated to the decomposition $M=\displaystyle\bigoplus_{i=1}^k M'_i$, then $\mathfrak{b'}=\mathfrak{b}$ and $\mathfrak{c'}=\mathfrak{c}$ by Theorem \ref{invariant sums}. Also, by Theorem \ref{invariant sums}, $$u_0\left(\displaystyle\bigoplus_{i=1}^k M_i\right)\hbox{ and }u_0\left(\displaystyle\bigoplus_{i=1}^k M'_i\right) \text{\;have the same image in\;}U_t.$$ So $\widetilde{\beta}(\mathfrak{b})=\widetilde{\beta}(\mathfrak{b'})$, $\widetilde{\beta}(\mathfrak{c})=\widetilde{\beta}(\mathfrak{c'})$, and by Lemma \ref{3}, $$\overline{\beta}\left(u_0\left(\displaystyle\bigoplus_{i=1}^k M_i\right)\right)\hbox{ and }\overline{\beta}\left(u_0\left(\displaystyle\bigoplus_{i=1}^k M'_i\right)\right) \text{\;have the same image in\;} U_t.$$ As $G(p^2)$ acts trivially on $\F_p$ we have that $$\overline{\beta}\left(u_0\left(\displaystyle\bigoplus_{i=1}^k M_i\right)\right)\hbox{ and }\overline{\beta}\left(u_0\left(\displaystyle\bigoplus_{i=1}^k M'_i\right)\right)\text{\;have the same image in\;}\F_p.$$ Therefore, by Theorem \ref{invariant sums}, the definition of $M^{\beta}$ does not depend on the choice of the decomposition of $M$ as a sum of indecomposables.

\begin{corol}\label{twisted repr}
    Let $M$ be a $\Z G$-lattice.  Then for $\beta \in G(p^2)$, there exists a $\Z$-isomorphism $$\Theta_\beta:M\longrightarrow M^{\beta},$$
such that $\Theta_\beta(g' \cdot m)=\beta(g')\cdot \Theta_\beta(m)\,\,\forall g' \in G, m \in M.$
\end{corol}

\begin{theorem} \label{M beta geral}
Let $M, M'$ be $\Z G$-lattices and $\beta\in G(p^2)$.  There is a $\Z$-isomorphism $\Psi_\beta:M\longrightarrow M'$ such that $\Psi_\beta(g'm)=\beta(g') \cdot \Psi_\beta(m)\,\, \forall g' \in G$ and $\forall m  \in M$ if, and only if, $M'\iso M^{\beta}$.

%and $\Psi_\beta:M\longrightarrow M'$ is a $\Z$-isomorphism such that $\Psi_\beta(gm)=\beta(g) \cdot \Psi_\beta(m),\; \forall g \in G$ and $\forall m  \in M$. Then $M'\iso M^{\beta}$.
\end{theorem}
\begin{proof}
    The backward implication is Corollary \ref{twisted repr}.  For the forward implication, let $\Theta_\beta:M\longrightarrow M^\beta$ be  the map given by  Corollary \ref{twisted repr}. Then $\Psi_\beta\Theta_\beta^{-1}$ is a $\Z G$-isomorphism.
\end{proof}

\section{Isomorphisms of semidirect products}

%$M=M(a,b,c,d,e; \beta_r, \gamma_r, \delta_r, \epsilon_r, \phi_r)$ and $M'=M(a',b',c',d',e'; \beta'_{r}, \gamma'_{r}, \delta'_{r}, \epsilon'_{r},\\ \phi'_{r})$
Let $M$ and $M'$ be two finitely generated free abelian groups and $G = C_{p^2}$. In this section we will decide when two semidirect products $M\rtimes C_{p^2}, M'\rtimes C_{p^2}$ (and their profinite completions) are isomorphic, which will be crucial for the calculation of the genus of such groups.

\begin{defn} \label{acao de G do semidireto}
Let $N$ be a normal (closed) subgroup of a (profinite) group $G.$ We will say that the (continuous) homomorphism $\psi: G \longrightarrow Aut(N)$ is induced by conjugation if $\psi(g) (n)=n^{g^{-1}}:= gng^{-1}, \, \forall g \in G \mbox{ and } n \in N.$    
\end{defn}

We restate three results from \cite{GZ} because they will be used repeatedly.

\begin{lemma}[{\cite[Lemma 2.1]{GZ}}] \label{representacoes equivalentes}
Let $N$ be a normal subgroup of the groups $G_1, G_2$ and let $f: G_1 \longrightarrow G_2$ be an isomorphism such that $f(N)=N.$ Let $\psi_i: G_i \longrightarrow Aut(N)\,(i=1,2)$ be the homomorphisms induced by conjugation and write $\bar f =f|_{_{N}}$ for the restriction of $f$ to $N.$ Then $[\psi_1(g)]^{(\bar f)^{-1}} = \psi_2(f(g))$ holds for all $g \in G_1.$ In particular, the images $\psi_1(G_1)$ and $\psi_2(G_2)$ in $Aut(N)$ are conjugated by the element $(\bar f)^{-1}.$      
\end{lemma}

% \begin{proof} Let $g \in G_1,$ then we have that   $\psi_1(g)(n)=n^{g^{-1}}$ and $\psi_2(f(g))(n)=n^{f(g)^{-1}},$ for all $n \in N.$ Therefore for all $n \in N$ we have $$[\psi_1(g)]^{(\bar f)^{-1}}(n)= \bar f \left( \psi_1(g) \left( f^{-1}(n) \right) \right)= \bar f \left( \left( f^{-1}(n) \right)^{g^{-1}} \right) =  n^{f(g^{-1})} = \psi_2(f(g))(n),$$ and the results follows.  
% \end{proof}

%The following results can be found on p. 142, Proposition 2.14 and Lemma 2.15, respectively in \cite{GZ}.  

\begin{prop}[{\cite[Proposition 2.14]{GZ}} or {\cite[Proposition 1.2 (6)]{Sam}}] \label{grupos no genero}
Let $M$ be a finitely generated free abelian group and let $H$ be a finite group. Let $\psi: H \longrightarrow Aut(M)$ be a homomorphism. Let $G$ be a finitely generated residually finite group with a profinite completion isomorphic to that of $
%G_{\psi}=
M \rtimes_{\psi} H.$ Then $G$ is isomorphic to a split extension $G \cong M \rtimes_{\eta} H.$      
\end{prop} 

\begin{lemma}[{\cite[Lemma 2.15]{GZ}}] \label{fiel}
Let $G$ be a group containing a torsion-free abelian group $M$ as a normal subgroup of finite index. Suppose that $M$ is a faithful $G/M$-module. Then every abelian subgroup of finite index in $G$ is contained in $M.$
\end{lemma}

\begin{prop} \label{i ii}
Let $M$ be  a finitely generated free abelian group and $\mathcal{E}=M \rtimes_{\psi_{}} G, \mathcal{E}'=M \rtimes_{\psi_{ }'} G$ be two semidirect products. Suppose that the actions of $G$ on $M$ induced by $\psi_{}$ and $\psi_{ }'$ are faithful. Denote by $M_{}$ and $M_{}'$ the respective $\Z G$-modules induced by the semidirect products. The groups $\mathcal{E}$ and $\mathcal{E} '$ are isomorphic if and only if $M_{} \cong (M_{}')^{\beta}$ as $\Z G$-modules for some $\beta \in G(p^2)$.
\end{prop}

\begin{proof}
Let $\Phi: \mathcal{E} \longrightarrow \mathcal{E}'$ be an isomorphism. Since the representation $\psi_{}$ is faithful, we have by Lemma \ref{fiel} that $\Phi(M)=M$. Define $\varphi  := \Phi|_{_{M}}$ and let $\beta$ be the isomorphism of $G$ given by $g\mapsto g'$, where $\Phi((0,g) + M) = (0,g')+M$.
%$\gamma^{-1}\Phi\gamma$, where $\gamma : G\longrightarrow$ $g\mapsto (g,0)+M\mapsto \Phi((g,0))+M$ induced by $\Phi$ through the isomorphism $\mathcal{E}/M \cong \mathcal{E}' /M \cong G.$ 
By \cite[Proposition 2.5]{GZ} the pair $(\varphi , \beta)$ satisfies $[\psi_{}(x)]^{\varphi ^{-1}} = \psi_{}'(\beta(x))$ for all $x \in G$ and yields an isomorphism from $\mathcal{E}$ to $\mathcal{E}'$ 
 given by $(\varphi, \beta)(m,x)=(\varphi(m), \beta(x)),$ $(m, x) \in \mathcal{E}.$ Recall that $x \cdot m = m^{x^{-1}} = \psi_{}(x)(m)$. Therefore we have that  \begin{equation} \label{conta semi abstr} \varphi(x \cdot m)= \varphi(\psi_{}(x)(m))= \left( \psi_{}'(\beta(x)) \circ \varphi\right)(m)=(\varphi(m))^{\beta(x)^{-1}}=\beta(x) \cdot \varphi(m),  \, \forall x \in G, m \in M_{},\end{equation} from which we conclude that $M_{} \cong (M_{}')^{\beta}$ as $\Z G$-modules, by Theorem \ref{M beta geral}.

Conversely, suppose we have a $\Z G$-module isomorphism $M_{} \iso (M_{}')^{\beta}$ for some $\beta \in G(p^2)$. By Theorem \ref{M beta geral} there is a $\mathbb{Z}$-isomorphism $\Upsilon: M' \longrightarrow M$ satisfying $\Upsilon(g' m')=\beta(g') \Upsilon(m'),$ $\forall g' \in G$ and $m' \in M'.$ Therefore \begin{equation} \label{conta abstra para o semi direto iso} \Upsilon \circ \psi'(g') \circ \Upsilon^{-1} = \psi(\beta(g')), \, \forall g' \in G.\end{equation} Define $\Omega: \mathcal{E}' \longrightarrow \mathcal{E}$ sending $(m',x') \in \mathcal{E}'$ to $(\Upsilon(m'), \beta(x'))$. It follows from \cite[Proposition 2.5 (1)]{GZ} that $\Omega$ is an isomorphism.
%$\mathcal{E}' \cong \mathcal{E}$.

% Conversely, suppose we have a $\Z G$-isomorphism $\Upsilon: M_{} \longrightarrow (M_{}')^{\beta}$ for some $\beta \in G(p^2)$.
% %$M_{\epsilon} \cong (M_{\epsilon'})^{\beta}$ as $\Z G$-modules, where $\beta \in G(p^2).$ Then the representations $\psi_{\epsilon}$ and $\psi_{\epsilon'}$ are $\Z$-equivalent, and so there exists a $\Z G$-isomorphism $\Upsilon: M_{\epsilon} \longrightarrow (M_{\epsilon'})^{\beta}$ such that $\Upsilon \circ \psi _{\epsilon}(x) \circ \Upsilon^{-1}= \psi_{\epsilon'}(\beta(x)), \, \forall x \in G.$ 
% Define $\Omega: \mathcal{E} \longrightarrow \mathcal{E}'$ sending $(m,x) \in \mathcal{E}$ to $(\Upsilon(m), \beta(x))$. It follows from 
% %Proposition 2.5 (i) of 
% \cite[Proposition 2.5]{GZ} that $\mathcal{E} \cong \mathcal{E} '$.  
 \end{proof}

%\begin{lemma} \label{maximal}
%Then $\widehat M$ is the unique open normal, torsion-free maximal abelian subgroup of $\widehat{\mathcal{E}}.$   
%\end{lemma}

%The Galois group $G(p^n)\,(n \in \mathbb{N})$ acts naturally on $\Q_p[\zeta_{p^n}]$ and therefore on $\Z_p[\zeta_{ p^n}]$ (see p. xx in \cite{neukirch}). In particular, $G(p^n)$ acts on the set of ideals of $\Z_p[\zeta_{ p^n}].$ The following result can be found in Exercise 6.2 of Chapter IV in \cite{bie} or 1 in \cite{neukirch}. 

%\begin{lemma} \label{galois e classes}
%Let $I$ and $J$ be ideals of $\Z_p[\zeta_{ p^n}].$ If $I$ and $J$ are in the same ideal class in $H(\Z_p[\zeta_{ p^n}]),$ then $\sigma(I)$ and $\sigma(J)$ are in the same ideal class for any $\sigma \in G(p^n)$ in other words $[I]=[J] \Longleftrightarrow [\sigma(I)]=[\sigma(J)] \in H(\Z_p[\zeta_{ p^n}]).$    
%\end{lemma}

\begin{lemma} \label{beta fixa invariantes} Let $M$ be a $\Z G$-lattice with genus determined by $(a,b,c,d,e; \beta, \gamma, \delta, \epsilon, \eta)$ 
%just as discussed in equations $(\ref{forma M}) - (\ref{forma MF})$ 
and $\beta \in G(p^2.)$ Then the map $\Theta_\beta:M\longrightarrow M^\beta$ 
 given by Corollary \ref{twisted repr} induces an automorphism of $\widehat{\Z}$-modules $\widehat{\Theta_{\beta}}:\widehat{M}\longrightarrow \widehat{M^\beta}$    satisfying  $\widehat{\Theta_{\beta}}(g'm)=\beta(g')\widehat{\Theta_{\beta}}(m)$ for $m\in \widehat{M}, g'\in G$. If $L$ is a $\widehat{\Z}G$-lattice possessing a $\widehat{\Z}$-automorphism $\Psi_\beta:\widehat{M}\longrightarrow L$ such that $\Psi_\beta(g'm)=\beta(g')\Psi_\beta(m)$ for $m \in \widehat{M}$ and $g'\in G$, then $L\iso \widehat{M^{\beta}}$. In particular, $L$ is the completion of a $\Z G$-lattice having the same genus invariants as $M$. 
 \end{lemma}

\begin{proof}
Since $M$ is finitely generated as a $\Z$-module, the system $Ch(M)$ of characteristic finite index
 subgroups of $M$ is cofinal in the system of all subgroups of finite index, because for any $n\in \N$, $n\cdot M$ is characteristic.  So $$\widehat{M}=\varprojlim_{N\in Ch(M)}M/N.$$
  Each $M/N$ is a $\Z G$-module since $N$ is characteristic, and in fact a $\widehat{\Z} G$-module because $l\Z$ acts as zero for some $l\in \N$.
  
  %furthermore since $N$ contains $lM$ for some $l\in \Z$  then $M/N$ is a $\Z/l\Z$-module in particular a module for $\widehat{\Z}$. 
  It follows that if $\Theta_\beta:M\longrightarrow M^\beta$ is the map given in  Corollary \ref{twisted repr}, then $\Theta_\beta$ induces a $\widehat{\Z}$-automorphism $\overline{\Theta_\beta}:M/N\longrightarrow M^{\beta}/N^{\beta}$. These maps yield a map of inverse systems as $N$ varies, and hence a $\widehat{\Z}$-automorphism $\widehat{\Theta_{\beta}}:\widehat{M}\longrightarrow \widehat{M^\beta}$ such that $\widehat{\Theta_{\beta}}(g'm)=\beta(g')\widehat{\Theta_{\beta}}(m)$ for $m\in \widehat{M}, g'\in G$. Now, suppose that $L$ is a $\widehat{\Z}G$-lattice possessing a $\widehat{\Z}$-automorphism $\Psi_\beta:\widehat{M}\longrightarrow L$ such that $\Psi_\beta(g'm)=\beta(g')\Psi_\beta(m)$ for $m \in \widehat{M}$ and $g'\in G$. Then $\widehat{\Theta}_\beta\Psi_\beta^{-1}: L\longrightarrow \widehat{M^{\beta}}$ is a $\widehat{\Z}G$-isomorphism.
\end{proof}

\begin{prop} \label{i ii profinite} Let $M$ be  a finitely generated free abelian group and 
$\mathcal{E}=M \rtimes_{\psi_{}} G,$  $\mathcal{E}'=M \rtimes_{\psi_{}'} G$ be two semidirect products. Suppose that the actions of $G$ on $M$ induced by $\psi_{}$ and $\psi_{}'$ are faithful. Denote by $M_{}$ and $M_{}'$ the respective $\Z G$-modules induced by the semidirect products.  Suppose that the genus invariants of $M_{}$ and $M_{}'$ respectively, are $(a,b,c,d,e; \beta, \gamma, \delta, \epsilon, \eta)$ and $(a',b',c',d',e'; \beta', \gamma', \delta', \epsilon', \eta')$.  The profinite completions $\widehat{\mathcal{E}}$ and $\widehat{\mathcal{E} '}$ are isomorphic if and only if $(a,b,c,d,e; \beta,  \epsilon, \eta)=(a',b',c',d',e'; \beta',  \epsilon', \eta')$ and $\gamma + \delta = \gamma'+\delta'$.
%$\gamma_r+\delta_r=\gamma'_r+\delta'_r$ for each $r \in \{1, \ldots, p-2\}.$
\end{prop}

\begin{proof}

%\medskip 

%\textcolor{red}{[[precisamos que $\widehat{M}\rtimes C_{p^2}\iso \prod(M_q\rtimes C_{p^2})$?  É verdade isso?]]}

Let $\Omega: \widehat{\mathcal{E}} \longrightarrow \widehat{\mathcal{E}}'$ be an isomorphism. It follows from \cite[Proposition 2.6]{GZ} that we can identify $\widehat{\mathcal{E}}=\widehat{M} \rtimes_{\widehat{\psi_{}}} G$ and $\widehat{\mathcal{E'}}=\widehat{M} \rtimes_{\widehat{\psi_{}'}} G$ (where $\widehat{\psi_{}}$ and $\widehat{\psi_{}'}$ are the natural homomorphisms discussed in \cite[pp.\ 136--137]{GZ}). Since the representations $\widehat{\psi}_{}$ and $\widehat{\psi}_{}'$ are faithful, it follows from Lemma \ref{fiel} that $\Omega(\widehat{M})=\widehat{M}.$ Just as in the abstract case it follows that $\Omega(x \cdot m)= \beta(x) \cdot \Omega(m),\, \forall x \in G, m \in \widehat{M}_{},$ where $\beta \in G(p^2)$ (as in Proposition \ref{i ii}). It follows from Lemmas \ref{posto e completamentos} and \ref{beta fixa invariantes} that $(a,b,c,d,e; \beta,  \epsilon, \eta)=(a',b',c',d',e'; \beta',  \epsilon', \eta')$ and $\gamma+\delta=\gamma'+\delta'$.  

Conversely, suppose that $M$ and $M'$ have the same genus invariants $(a,b,c,d,e; \beta,  \epsilon, \eta)$ and that $\gamma+\delta=\gamma'+\delta'$.  By Lemma \ref{posto e completamentos}, $\widehat{M_{} } \cong \widehat{M_{}'}$ as $\widehat{\Z} G$-lattices and so $\widehat{\mathcal{E}}\iso \widehat{\mathcal{E}'}$. % Therefore the representations $\widehat{\psi_{}}$ and $\widehat{\psi_{}'}$ are $\widehat{\Z}$-equivalent, and so there exists a $\widehat{\Z} G$-isomorphism $\Upsilon: \widehat{M_{}} \longrightarrow \widehat{M_{}'}$ such that $\Upsilon \circ \widehat{\psi _{}}(x) \circ \Upsilon^{-1}= \widehat{\psi_{}'}(x), \, \forall x \in G.$ Define $\Omega: \widehat{\mathcal{E}} \longrightarrow \widehat{\mathcal{E}'}$ such that $\Omega$ sends $(m,x) \in \widehat{M} \times G$ to $(\Upsilon(m),x) \in \widehat{\mathcal{E'}}=\widehat{M} \rtimes_{\widehat{\psi_{}'}} G.$ Note that $\Omega$ is a homomorphism because if $x_1, x_2 \in G$ and $m_1, m_2 \in \widehat{M},$ then we have \begin{equation*} \label{conta semi} \Omega\left( (m_1, x_1) \cdot (m_2, x_2)\right)= \Omega(m_1 +m_2^{x_1^{-1}}, x_1x_2)=\left(  \Upsilon(m_1) +\Upsilon(x_1 \cdot m_2), x_1x_2 \right) = \end{equation*} \begin{equation*}
%=\left(  \Upsilon(m_1) + (x_1 \cdot \Upsilon( m_2)), x_1x_2 \right) =\left(  \Upsilon(m_1) +\Upsilon( m_2)^{x_1^{-1}}, x_1x_2 \right)=\Omega(m_1, x_1) \cdot \Omega(m_2, x_2),\end{equation*} and it is trivially bijective because $\Upsilon$ is, so $\widehat{\mathcal{E}} \cong \widehat{\mathcal{E'}}.$

%Note that $[\widehat{\psi}_{\epsilon}(x)]^{\Phi_{\widehat M}^{-1}} = \widehat{\psi}_{\epsilon '}(\Theta(x))$ holds for all $x \in C_{p^2}.$ Therefore $( \Phi_{\widehat M}, \Theta): \widehat{\mathcal{E}} \longrightarrow \widehat{\mathcal{E}}'$ defined by $( \Phi_{\widehat M}, \Theta)(m, x)=(\Phi_{\widehat{M}}(m), \Theta(x)), m \in \widehat{M}, x \in C_{p^2}$ is an isomorphism between $\widehat{\mathcal{E}}$ and $\widehat{\mathcal{E}} '$ by Proposition 2.5 in \cite{GZ}.  
\end{proof}

\begin{theorem}\label{theorem sd products in same genus iff modules in same orbit}
    Given $\mathcal{E} = M\rtimes_\psi G$, the value $|\mathfrak{g}(\mathcal{E})|$ is equal to the number of distinct orbits of the action of $G(p^2)$ on the genus of $M$ as a $\Z G$-module.
\end{theorem}

\begin{proof}
    Direct from Propositions \ref{grupos no genero}, \ref{i ii} and \ref{i ii profinite}.
\end{proof}

%%%%%%%%%%%%%%%%%%%%% seção nova profinite genus

\section{Profinite genus of semi-direct products}\label{section calculation of profinite genus}

In this section, to economize notation, by ``faithful semidirect product'' we mean a semidirect product $\mathcal{E}=M \rtimes_{\psi_{}} G$, where $G = C_{p^2},$ $M$ is a finitely generated free abelian group and the action of $G$ on $M$ is faithful.  We will denote also by $M$ the $\Z G$-lattice induced by $\psi$.

%In all the results of this section we will assume the validity of the following condition :\\

%\textbf{Condition I):} Consider $\mathcal{E}=M \rtimes_{\psi_{}} G$ a semidirect product of a finitely generated free abelian group by a cyclic group $G$ of order $p^2,$ where $p$ is a prime number. We will denote by $M_{}$ the $\Z G$-lattice induced by $\psi_{}$ and assume that this action is always faithful. 

We will calculate the cardinality of the genus of $\mathcal{E}$ in the class finitely generated residually finite groups. 
%i.e we want to find a formula for the expression $|\mathfrak{g}\left( \mathcal{E}, \mathfrak{RF} \right)|.$ 
However, as we will see below, the size of the genus depends on the structure of the $\Z G$-module $M_{}$, so we will divide the work into cases.  

% and we will have to divide the calculation of this genus into several results with certain restrictions for $M_{}.$  

\begin{remark}

When the action of $G$ on $M$ is not faithful, then $M_{}$ is a $\Z C_p$-lattice of the form $M_{}=\Z^{a} \oplus \displaystyle \bigoplus_{i=1}^{b} \, \mathfrak{b}_i \oplus  \bigoplus_{k=1}^{d} E(\mathfrak{b}'_k).$ If $M_{}$ is nontrivial, $ |\mathfrak{g}\left(\mathcal{E}\right)|= \left| G(p) \backslash \, H(\mathbb{Z}[\zeta_{p}])  \right|$ by \cite[Proposition 2.23]{GZ}).  On the other hand, if $M_{}$ is a trivial module then the cardinality of genus is equal to 1 since $\mathcal{E}=M \times G \cong \Z^a \times C_{p^2}$ is a finitely generated abelian group and such groups are profinitely rigid, see \cite[Proposition 3.1]{reid}).     
\end{remark}

Recall that $H(\Z[\zeta_{p^n}])$ denotes the ideal class group of the field extension $\Q(\zeta_{p^n}):\Q$.

\begin{prop} \label{so Cs}
Let $\mathcal{E}=M \rtimes_{\psi_{}} G$ be a faithful semidirect product. %Let $\mathcal{E}=M \rtimes_{\psi_{\epsilon}} G$ be a semidirect product  of a finitely generated free abelian group by a cyclic group $G$ of order $p^2,$ where $p$ is a prime number. Denote by $M_{\epsilon}$ this $\Z G$-lattice induced by $\psi_{\epsilon}$ and suppose this action is faithful. 
Suppose 
$$M_{}=\Z^{a} \oplus \displaystyle \bigoplus_{j=1}^{c} \, \mathfrak{c}_j  \oplus \displaystyle \bigoplus_{l=1}^{e} \, E(\mathfrak{c}'_l),$$ 
with $a \geqslant 0$ and $c+e \geqslant 1.$   Then $$|\mathfrak{g}\left(\mathcal{E}\right)|= \left| G(p^2) \backslash \, H(\mathbb{Z}[\zeta_{p^2}])  \right|.$$ 
%Let $M_{\epsilon}$ be the $\Z G$-lattice induced by the faithful action of $G$ which is induced by $\psi_{\epsilon}.$ 
%Suppose $M_{\epsilon}$ is a $\Z G$-lattice satisfying  sum of indecomposables has a summand of the types $\kb$ or $\kc$ or $E(\kb)$ or $E(\kc)$ and either  $p \not\equiv \, 1 (\mbox{mod } \, 4 )$ or $p \equiv \, 1 (\mbox{mod } \, 4 )$ and among the summations that form $M$ there is at least one of the types $\Z$ or $\kb, \kc$ or $ (E(\kb), \kc; \lambda^{r} u )$ or $(\Z \oplus \kb, \kc; 1 \oplus \lambda^{r} u ).$ Then $$|\mathfrak{g}\left(\mathcal{E}\right)|= \left| G(p) \backslash \, H(\mathbb{Z}[\zeta_{p}])  \right|\cdot \left| G(p^2) \backslash \, H(\mathbb{Z}[\zeta_{p^2}])  \right|.$$ 
\end{prop}

\begin{proof}
Let $\mathcal{E}'$ be a finitely generated residually finite group such that $\widehat{\mathcal{E}} \cong \widehat{\mathcal{E}}'.$ By Proposition \ref{grupos no genero}, $ \mathcal{E} ' = M \rtimes_{\psi_{} '} G.$ Denote by $M_{}'$ the $\Z G$-lattice induced by $\psi_{}'.$ By Proposition \ref{i ii profinite}, $M_{}'=\Z^{a} \oplus \bigoplus_{j=1}^{c} \, \mathfrak{c}''_j  \oplus \bigoplus_{l=1}^{e} \, E(\mathfrak{c}'''_l)$. By Proposition \ref{i ii}, $\mathcal{E} \cong \mathcal{E}'$ if and only if $M_{} \cong (M_{}')^{\beta}$ for some $ \beta \in G(p^2).$ Therefore, when $\mathcal{E}\iso \mathcal{E}'$, we have that $\displaystyle\prod_{j=1}^{c}\beta(\mathfrak{c}''_j)\cdot \displaystyle\prod_{l=1}^{e} \, \beta( \mathfrak{c}'''_l)$ and $\displaystyle\prod_{j=1}^{c}\mathfrak{c}_j\cdot \displaystyle\prod_{l=1}^{e} \,  \mathfrak{c}'_l$ represent the same element of $H(\mathbb{Z}[\zeta_{p^2}])$, so that $\displaystyle\prod_{j=1}^{c}\mathfrak{c}''_j\cdot \displaystyle\prod_{l=1}^{e} \,  \mathfrak{c}'''_l$ and $\displaystyle\prod_{j=1}^{c}\mathfrak{c}_j\cdot \displaystyle\prod_{l=1}^{e} \,  \mathfrak{c}'_l$  are  
in the same $G(p^2)$-orbit of $H(\mathbb{Z}[\zeta_{p^2}]).$ Since the only isomorphism invariants needed to describe the isomorphism class of $M_{}$ are those of Items i) and ii) of Theorem \ref{invariant sums},    it follows  that the profinite genus depends only on the orbits of the action of $G(p^2)$ on $H(\mathbb{Z}[\zeta_{p^2}])$, so by Theorem \ref{theorem sd products in same genus iff modules in same orbit} it follows that $$\left|\mathfrak{g}\left( M \rtimes_{\psi_{}} G\right)\right|= \left| G(p^2) \backslash \, H(\mathbb{Z}[\zeta_{p^2}]) \right|.$$ 
\end{proof}

\begin{prop}\label{Cs e Bs com absortion}
Let $\mathcal{E}=M \rtimes_{\psi_{}} G$ be a faithful semidirect product and suppose that $p \not\equiv \, 1 (\mbox{mod } \, 4 ).$ Suppose \begin{equation*}  M_{}= \Z^{a} \oplus \displaystyle \bigoplus_{i=1}^{b} \, \mathfrak{b}_i \oplus  \bigoplus_{j=1}^{c} \, \mathfrak{c}_j  \oplus \bigoplus_{k=1}^{d} E(\mathfrak{b}'_k) \oplus \bigoplus_{l=1}^{e} \, E(\mathfrak{c}'_l)  \oplus \underbrace{M_B \oplus M_C \oplus M_E \oplus M_F}_{M_Q}\end{equation*} 
is such that (at least) two of the following conditions are satisfied: (1) $ b+d  \geqslant 1$, (2) $ c+e  \geqslant 1$, (3) $M_Q\neq 0$.
%one of the following conditions is satisfied: (i) $ b+d  \geqslant 1,$ $M_Q=\{0\}$ and $c+e \geqslant 1;$ or (ii) $b+d \geqslant 1,$ $M_Q \neq \{0\}$; or (iii) $c+e \geqslant 1$ and $M_Q \neq \{0\}.$ 
Then $$|\mathfrak{g}\left(\mathcal{E}\right)|= \left| G(p) \backslash \, H(\mathbb{Z}[\zeta_{p}])  \right| \cdot \left| G(p^2) \backslash \, H(\mathbb{Z}[\zeta_{p^2}])  \right|.$$ 
\end{prop}

\begin{proof}  The proof is similar to the previous result.  Let $\mathcal{E}' = M\rtimes_{\psi'}G$ be a semidirect product in the same profinite genus as $\mathcal{E}$, and denote by 
\begin{equation} \label{segundo caso}  M_{}'= \Z^{a} \oplus \displaystyle \bigoplus_{i=1}^{b} \, \mathfrak{b}''_i \oplus  \bigoplus_{j=1}^{c} \, \mathfrak{c}''_j  \oplus \bigoplus_{k=1}^{d} E(\mathfrak{b}'''_k) \oplus \bigoplus_{l=1}^{e} \, E(\mathfrak{c}'''_l)  \oplus \underbrace{M'_B \oplus M'_C \oplus M'_E \oplus M'_F}_{M'_Q}\end{equation} 
the corresponding $\Z G$-module.  By Proposition \ref{i ii profinite}, $M'$ also satisfies the conditions of the statement. By Proposition \ref{i ii}, $\mathcal{E} \cong \mathcal{E}'$ if, and only if $M_{} \cong (M_{}')^{\beta}$ for some $ \beta \in G(p^2).$ Denote by $\mathfrak{b}$ and
by $\mathfrak{c}$ the ideal classes of $R
$ and $S$ associated to $M$ and by $\mathfrak{b}'$ and $\mathfrak{c}'$ the ideal classes of $R$ and $S$ associated to $M'$. Therefore when $\mathcal{E}\iso \mathcal{E}'$, we have that $\beta(\mathfrak{b'})=\mathfrak{b}$ and $\beta(\mathfrak{c'})=\mathfrak{c}.$
 So $\mathfrak{b}'$ and $\mathfrak{b}$ 
  are equal in $G(p^2) \backslash \,H(\mathbb{Z}[\zeta_{p}])=G(p) \backslash \,H(\mathbb{Z}[\zeta_{p}])$ (the obvious homomorphism $G(p^2) \longrightarrow G(p)$ is surjective).
  Similarly, $\mathfrak{c}'$ and $\mathfrak{c}$ are equal in $G(p^2) \backslash \,H(\mathbb{Z}[\zeta_{p^2}]).$
 When at least two of conditions (1), (2), (3) are satisfied, the only isomorphism invariants needed to describe the class of $M_{}$ are as in items i) and ii) of Theorem \ref{invariant sums}. This shows that the genus depends only on the orbits of the actions of $G(p^2)$ on $ H(\mathbb{Z}[\zeta_{p^2}])$ and of $G(p^2)$  on $H(\mathbb{Z}[\zeta_{p}]),$ respectively, so by Theorem \ref{theorem sd products in same genus iff modules in same orbit} it follows that $$\left|\mathfrak{g}\left( M \rtimes_{\psi_{}} G\right)\right|= \left| G(p) \backslash \, H(\mathbb{Z}[\zeta_{p}]) \right| \cdot \left| G(p^2) \backslash \, H(\mathbb{Z}[\zeta_{p^2}]) \right|.$$ 
%since both modules do not contain summands of types (C) and (D) (the character only gives $-1$ if $p \equiv \, 1 (\mbox{mod } \, 4 )$ and there is an odd number of summands of type (D)). By Theorem \ref{invariant sums} note that the only isomorphism invariants needed to describe the class of $M_{\epsilon}$ are as in items i) and ii) of the above theorem. This shows that the genus depends only on the orbits of the actions of $G(p^2)$ on $H(p^2)$ and of $G(p^2)$ (also $G(p)$) on $H(p),$ respectively, whence it follows that $$\left|\mathfrak{g}\left( M \rtimes_{\psi_{\epsilon}} G, \mathfrak{RF}\right)\right|= \left| G(p) \backslash \, H(\mathbb{Z}[\zeta_{p}]) \right| \cdot \left| G(p^2) \backslash \, H(\mathbb{Z}[\zeta_{p^2}]) \right|.$$  
\end{proof} 

\begin{example}
    Consider $M$ and $\mathcal{E}$ in the conditions of either Proposition \ref{so Cs} with $p\in \{2,3,5\}$ or of Proposition \ref{Cs e Bs com absortion} with $p\in \{2, 3\}$.  Then \cite[Theorem 11.1]{wash} tells us that $h(\mathbb{Z}[\zeta_{p}])=h(\mathbb{Z}[\zeta_{p^2}]) = 1$, so by the propositions, $\mathcal{E}$ is profinitely rigid.
\end{example}

% \begin{corol} \label{cor 22}
% If $M$ and $\mathcal{E}$ are as in the conditions of either Proposition \ref{so Cs} or Proposition \ref{Cs e Bs com absortion} %\ref{ Cs e Bs com absortion}, 
% and $p\in \{2, 3, 5\}$, then $\mathcal{E}$ is profinitely rigid.% $|\mathfrak{g}(\mathcal{E})|=1.$ 
% \end{corol}

% \begin{proof}
% It follows directly from the previous propositions together with \cite[Theorem 11.1]{wash} that $h(\mathbb{Z}[\zeta_{p}])=h(\mathbb{Z}[\zeta_{p^2}]) = 1$ for the $p$ in the statement, 
% %$$h(\mathbb{Z}[\zeta_{2}])=h(\mathbb{Z}[\zeta_{3}])=h(\mathbb{Z}[\zeta_{5}])=h(\mathbb{Z}[\zeta_{4}])=h(\mathbb{Z}[\zeta_{9}])=h(\mathbb{Z}[\zeta_{25}])=1,$$ 
% so there is a single orbit in each of the actions of the Galois groups. 
% \end{proof}

\begin{example}
    If $M$ and $\mathcal{E}$ are as in Proposition \ref{Cs e Bs com absortion} and $p\in\{7,11,19\}$, then \cite[Theorem 11.1]{wash} tells us that $h(\mathbb{Z}[\zeta_{p}])=1$, so that $\left| G(p^2) \backslash \, H(\mathbb{Z}[\zeta_{p}])  \right|=\left| G(p) \backslash \, H(\mathbb{Z}[\zeta_{p}])  \right|=1$.  It follows from the proposition that $|\mathfrak{g}(\mathcal{E})|= \left| G(p^2) \backslash \, H(\mathbb{Z}[\zeta_{p^2}])  \right|$.
\end{example}

% \begin{corol} \label{cor 22b}
% If $M$ and $\mathcal{E}$ are as in the conditions of Proposition \ref{Cs e Bs com absortion} and $7 \leqslant p \leqslant 19$ is a prime number, then $|\mathfrak{g}(\mathcal{E})|= \left| G(p^2) \backslash \, H(\mathbb{Z}[\zeta_{p^2}])  \right|.$
% %$|\mathfrak{g}(\mathcal{E})|=1.$ 
% \end{corol}

% \begin{proof}
% It follows directly from Proposition \ref{Cs e Bs com absortion} together with \cite[Theorem 11.1]{wash} that $h(\mathbb{Z}[\zeta_{p}])=1$ for the $p$ in the statement, 
% %$$h(\mathbb{Z}[\zeta_{2}])=h(\mathbb{Z}[\zeta_{3}])=h(\mathbb{Z}[\zeta_{5}])=h(\mathbb{Z}[\zeta_{4}])=h(\mathbb{Z}[\zeta_{9}])=h(\mathbb{Z}[\zeta_{25}])=1,$$ 
% so $\left| G(p^2) \backslash \, H(\mathbb{Z}[\zeta_{p}])  \right|=\left| G(p) \backslash \, H(\mathbb{Z}[\zeta_{p}])  \right|=1$ and the result follows. 
% \end{proof} 

\begin{example}
Consider $p=7$ and to simplify notation set $\zeta = \zeta_{49}$.  We will calculate that $$\left| G(49) \backslash \, H(\mathbb{Z}[\zeta ]) \right|=2,$$ so when $\mathcal{E}$ is as in Proposition \ref{Cs e Bs com absortion}, then $|\mathfrak{g}(\mathcal{E})|=2.$ By \cite[Theorem 2.5]{wash}, $G(49) \cong C_{42}$ and $H(\mathbb{Z}[\zeta ]) \cong C_{43}.$ The automorphism $\delta$ of $\mathbb{Q}[\zeta ]$ that sends $\zeta $ to $\zeta ^3$ generates the Galois group $G(49)$.  Let $[I]$ be a generator of $H(\mathbb{Z}[\zeta])$. The group $G(49)$ acts on $H(\mathbb{Z}[\zeta ])$ via the homomorphism $G(49) \longrightarrow\mbox{Aut}(H(\mathbb{Z}[\zeta ]))$ which sends $\delta$ to the automorphism of $H(\mathbb{Z}[\zeta ])$ sending $[I]$ to $[I]^3.$ Then $\delta^k([I]) = [I]^{3^k}$, so since $[I]^3$ generates $H(\mathbb{Z}[\zeta])$ and $|\delta| = 42$, we have $|G(49)\cdot [I]|=42$. It follows that $H(\mathbb{Z}[\zeta])$ has two orbits under the action of $G(49)$.  From the proof of Proposition \ref{Cs e Bs com absortion}, the two groups in the genus of $\mathcal{E}$ correspond to $M$ in the statement having $S$-ideal class either trivial, or non-trivial. 
\end{example}

In that follows, $U_t$ is as in Theorem \ref{invariant sums}.

\begin{prop} \label{sem absorcao e sem D}
Let $\mathcal{E}=M \rtimes_{\psi_{}} G$ be a faithful semidirect product 
 and $p \not\equiv \, 1 (\mbox{mod } \, 4 ).$  
Suppose \begin{equation} \label{p not 1 e so C} M_{}= \Z^{a} \oplus  \underbrace{M_B \oplus M_{C}  \oplus M_E \oplus M_F}_{M_Q}, \end{equation} with  $M_Q \neq \{0\}.$ Then $$|\mathfrak{g}\left(\mathcal{E}\right)|= \left| G(p) \backslash \, H(\mathbb{Z}[\zeta_{p}])  \right| \cdot \left| G(p^2) \backslash \, H(\mathbb{Z}[\zeta_{p^2}])  \right| \cdot \left| G(p^2) \backslash \, U_t  \right|.$$ 
\end{prop}

\begin{proof} 
  The presence of the factor $\left| G(p) \backslash \, H(\mathbb{Z}[\zeta_{p}])  \right| \cdot \left| G(p^2) \backslash \, H(\mathbb{Z}[\zeta_{p^2}])  \right|$ is similar to the preceding propositions. We will discuss the presence of the third factor only. 
  By Proposition \ref{i ii}, $\mathcal{E} \cong \mathcal{E}'$ if and only if $M_{} \cong (M_{}')^{\beta}$ for some $ \beta \in G(p^2).$  By Lemma \ref{3} we have $u_0\left( \left( M_{}'\right)^{\beta} \right)= \overline\beta(u_0(M_{}'))$, but  $M_{} \cong (M_{}')^{\beta}$ and thus by Theorem \ref{invariant sums}, $u_0((M')^{\beta})$ and $u_0(M_{})$ have the same image in $U_t$. %  and so  we must have $u_0((M')^{\beta})=u_0(M_{})$ in $U_t$. 
  Therefore when $\mathcal{E}\iso \mathcal{E}'$, then the images of $u_0\left(M_{}'\right)$ and $u_0(M_{})$ are in the same orbit of the action of $G(p^2)$ on $U_t$.  
\end{proof}

The cases that still need to be analyzed correspond to situations in which $p \equiv \, 1 (\mbox{mod } \, 4 ).$ 

%The reason for the following definition will become clear in the proofs of the propositions below, however, it is worth mentioning that the only way for $2$ to occur in the previous definition is if $p \equiv \, 1 (\mbox{mod } \, 4 )$ because this way there could be a direct summands of type (D). 

\begin{defn} \label{sigma} 
Let $M$ be a $\Z G$-lattice.  Define $\Sigma_M$ to be $2$ if $p \equiv \, 1 (\mbox{mod } \, 4 )$ and $M$ has at least one summand of type $(C)$ or $(D)$ and no summand of the form $\Z$, $E(\mathfrak{b})$, $E(\mathfrak{c})$ or of type $(B)$ or $(F)$.  In all other cases define $\Sigma_M$ to be $1$.
%If $p \not \equiv \, 1 (\mbox{mod } \, 4 )$ define $\Sigma_M=1.$ %When $p \equiv \, 1 (\mbox{mod } \, 4 ),$ define $\Sigma_M=1$ if the $\Z G$-lattice $M_{}$ has no direct summands of type $(C)$ or $(D)$ or if it has at least one direct summand of the form $\Z$, $E(\mathfrak{b})$,$E(\mathfrak{c})$ or type $(B)$ or $(F)$.  Otherwise $\Sigma_M=2.$ 
%Define $\Sigma_M=2$ if $p \equiv \, 1 (\mbox{mod } \, 4 )$ and the $\Z G$-lattice $M_{}$ has at least one summand of type $(C)$ or $(D)$ and no summand of the form $\Z$, $E(\mathfrak{b})$, $E(\mathfrak{c})$ or type $(B)$ or $(F)$.  Otherwise $\Sigma_M = 1$.
\end{defn}

\begin{prop} \label{mais simples congruo a 1} Let $\mathcal{E}=M \rtimes_{\psi_{}} G$ be a faithful semidirect product and $p \equiv \, 1 (\mbox{mod } \, 4 ).$
   Suppose \begin{equation} \label{penultimo eq} M_{}= \Z^{a} \oplus  \underbrace{M_B \oplus M_{C,D}  \oplus M_E \oplus M_F}_{M_Q}, \end{equation} 
   with $M_Q\neq \{0\}$.
 %  such that any of the following conditions is satisfied: \begin{enumerate}
%   \item
 %   $a \geqslant 1$ and $M_Q \neq \{0\},$ \item  $M_B \oplus M_F \neq \{0\},$ \item $a=0, M_B \oplus M_F = \{0\}$ and $M_{C,D} \oplus M_E \neq \{0\}.$\end{enumerate}
     Then $$|\mathfrak{g}\left(\mathcal{E}\right)|= \left| G(p) \backslash \, H(\mathbb{Z}[\zeta_{p}])  \right| \cdot \left| G(p^2) \backslash \, H(\mathbb{Z}[\zeta_{p^2}])  \right| \cdot \left| G(p^2) \backslash \, {U}_t  \right| \cdot \Sigma_M.$$ 
\end{prop}

\begin{proof} %The proof of this proposition is similar to that of the previous ones and we enter into detail only when $\Sigma_M = 2$.  
In most cases the proof is as in the previous propositions, because invariant (iv) from Theorem \ref{invariant sums} need not be considered.  A new case occurs when $M_{}=M_{C,D} \oplus M_E$ with $M_{C,D} \neq \{0\}$.
Let $\mathcal{E} '$ be a finitely generated residually finite group in the same genus as $\mathcal{E}$. % such that $\widehat{\mathcal{E}} \cong \widehat{\mathcal{E}}'.$ 
By Proposition \ref{grupos no genero}, $ \mathcal{E} ' = M \rtimes_{\psi_{} '} G.$ Denote by $M_{}'$ the $\Z G$-lattice induced by $\psi_{}'.$   Let $(0,0,0,0,0; 0, \gamma, \delta, \epsilon, 0)$ and  $(0,0,0,0,0; 0, \gamma', \delta', \epsilon', 0)$  be the parameters defining the genus of $M,M'$ respectively (cf.\ (\ref{forma M}) - (\ref{forma MF})).
%We will write $M_{}$ and $M_{}'$ respectively as $M(a,b,c,d,e; \beta, \gamma, \delta, \epsilon, \eta)$ and $M(a',b',c',d',e'; \beta', \gamma', \delta', \epsilon', \eta')$ as discussed in the equations $(\ref{forma M}) - (\ref{forma MF}).$ 
By Lemma \ref{beta fixa invariantes} and Proposition \ref{i ii profinite} we have that 
$$\epsilon = \epsilon' \mbox{ and } \gamma_r+\delta_r=\gamma'_r+\delta'_r \mbox{ for each } r \in \{1, \ldots, p-2\}.$$ 
We can write $$M=M_{C,D}\oplus M_E=(\Z\oplus E(\mathfrak{b}), \mathfrak{c}; 1\oplus \lambda^run_0^k)\oplus X_M,$$ where $\mathfrak{b}$ and $\mathfrak{c}$ are the $R$ and $S$-ideal classes associated to $M$ with $u=u_0(M), k\in \{0,1\}$ and $X_M$ is a direct sum of modules of type $(\Z\oplus E, S;  1\oplus\lambda^{r'})$ or $(R, S; \lambda^{r''})$.  As the parameters defining the genus of $M$ and $M'$ are equal except possibly for $\gamma$ and $\delta$, we can also write $$M' = M'_{C,D} \oplus M'_E=(\Z\oplus E(\mathfrak{b}'), \mathfrak{c}'; 1\oplus\lambda^{r}u'n_0^{k'})\oplus X_{M'}$$ where $\mathfrak{b}'$ and $\mathfrak{c}'$ are the $R$ and $S$-ideal classes associated to $M'$ and $u'=u_0(M'), k'\in \{0,1\}$ and $X_{M'}$ is a direct sum of modules of type $(\Z\oplus E, S; 1\oplus\lambda^{r'})$ or $(R, S;  \lambda^{r''})$. Of course, we have $X_M=X_{M'}$. Note that $(\Z\oplus E(\mathfrak{b}),\mathfrak{c}; 1\oplus\lambda^run_0^k)$ and $(\Z\oplus E(\mathfrak{b}'), \mathfrak{c}'; 1\oplus\lambda^{r}u'n_0^{k'})$ are in the same genus, but $k$ and $k'$ might not have the same parity. If the parities are different, then $\mathcal{E}\not\cong \mathcal{E'}$ by Proposition \ref{i ii}. Without loss of generality we can suppose that $k=0$.  Given another module $M''=(\Z\oplus E(\mathfrak{b}''), \mathfrak{c}''; 1\oplus\lambda^{r}u'' ) \oplus X_{M}$ in the genus of $M$, it follows that
%So the modules $M'=(\Z\oplus E(\mathfrak{b}'), \mathfrak{c}'; \lambda^{r}u'n_0) \oplus X_{M}$ and $M''=(\Z\oplus E(\mathfrak{b}''), \mathfrak{c}''; \lambda^{r}u'' ) \oplus X_{M}$ are in the same genus, but 
$\mathcal{E}$ is isomorphic to $\mathcal{E}''=M'' \rtimes_{\psi''} G$ if only if there exists $\beta \in G(p^2)$ such that $M_{} \cong (M_{}'')^{\beta}.$ So the genus of $M$ splits into two parts, one in which the modules have the form $M'=(\Z\oplus E(\mathfrak{b}'),\mathfrak{c}'; 1\oplus\lambda^{r}u'n_0) \oplus X_{M}$ and another in which the modules have the form $M''=(\Z\oplus E(\mathfrak{b}''),\mathfrak{c}''; 1\oplus\lambda^{r}u'' ) \oplus X_{M}.$ %We have shown that the value  of $|\mathfrak{g}(\mathcal{E})|$ corresponds to the number of distinct orbits of the action of $G(p^2)$ on the genus of $M$.
Since the two parts are invariant under the action of $G(p^2)$, the size of the genus of $M$ is the sum of the size of the genuses of the two parts, so by Theorem \ref{theorem sd products in same genus iff modules in same orbit}, the size of the genus of $\mathcal{E}$ is equal to the sum of the number of distinct $G(p^2)$-orbits of each part. The decomposition of $M$ gives that it is sufficient to analyze the number of orbits of the action of $G(p^2)$ on the genus of modules  of type $(\Z\oplus E ,S; 1\oplus \lambda^{r})$. But this number is exactly $$\left| G(p) \backslash H(\mathbb{Z}[\zeta_{p}]) \right|\left| G(p^2) \backslash H(\mathbb{Z}[\zeta_{p^2}) \right|\left| G(p^2) \backslash {U}_t \right|,$$ 
so $$|\mathfrak{g}\left(\mathcal{E}\right)|= 2 \cdot \left| G(p) \backslash \, H(\mathbb{Z}[\zeta_{p}])  \right| \cdot \left| G(p^2) \backslash \, H(\mathbb{Z}[\zeta_{p^2}])  \right| \cdot \left| G(p^2) \backslash \, {U}_t  \right|,$$
where the factor $2$ appears because of the two parts of the genus of $M$.
\end{proof}

% \begin{example}
% We can use Proposition \ref{mais simples congruo a 1} to construct non-trivial examples of groups with small genus.  For instance, if $p=5$ and $M = \bigoplus_{i=1}^n(\Z\oplus E(\mathfrak{b}_i), \mathfrak{c}_i; 1\oplus \lambda^3)$ then the profinite genus of  $\mathcal{E} = M\rtimes G$ is $2$, because $U_1=1$.   If $p = 5$ and $ M = \bigoplus_{i=1}^n\left( E(\mathfrak{b}_{i}), \mathfrak{c}_{i}; \lambda^4\right)$, then the profinite genus of  $\mathcal{E} = M\rtimes G$ is $1$, because $U_1=1$.
% \end{example}

\begin{example}
We can use Proposition \ref{mais simples congruo a 1} to construct non-trivial examples of groups with small genus.  For instance, if $p=5$ and $M = \bigoplus_{i=1}^n(\Z\oplus E(\mathfrak{b}_i), \mathfrak{c}_i; 1\oplus \lambda^3)$ then the profinite genus of  $\mathcal{E} = M\rtimes G$ is 2, while if $ M = \bigoplus_{i=1}^n\left( E(\mathfrak{b}_{i}), \mathfrak{c}_{i}; \lambda^4\right)$, then the profinite genus of  $\mathcal{E} = M\rtimes G$ is $1$. This occurs because in the first case we have  $\Sigma_M =
2$ and in the second case $\Sigma_M = 1$. Also, in both cases  $U_t = U_1 = 1$ (see \cite[p.\ 740] {curtis}), and \cite[Theorem 11.1]{wash} tells us that $h(\mathbb{Z}[\zeta_{5}])=h(\mathbb{Z}[\zeta_{25}]) = 1$.
\end{example}

% \begin{example}
% Suppose that $M$ and $\mathcal{E}$ are as in Proposition \ref{mais simples congruo a 1}, with $p=5$ and $|G(p^2)\backslash U_t|=1.$ Then $|\mathfrak{g}(\mathcal{E})|=1$ if $M_{C,D}=\{0\},$ and we have  $|\mathfrak{g}(\mathcal{E})|=2$ if $M_{C,D} \neq \{0\}$ and $M$ does not have summand of the form $\Z$ or of type $(B)$ or $(F)$.  
% \end{example}

\begin{prop} \label{com b e c mas sem valer iv}
  Let $\mathcal{E}=M \rtimes_{\psi_{}} G$ be a faithful semidirect product and $p \equiv \, 1 (\mbox{mod } \, 4 ).$
  Suppose \begin{equation} M_{}= \displaystyle \bigoplus_{i=1}^{b} \, \mathfrak{b}_i \oplus  \bigoplus_{j=1}^{c} \, \mathfrak{c}_j \oplus M_{C,D} \oplus M_E\end{equation} is such that (at least) one of the following conditions is satisfied: 
  \begin{enumerate}
      %\item $b \geqslant 1$ and ($c \geqslant 1$ or $M_{C,D} \oplus M_E \neq \{0\})$,
      \item $b \geqslant 1$ and $c \geqslant 1$,
      \item $b+c \geqslant 1$ and $M_{C,D} \oplus M_E \neq \{0\}.$
  \end{enumerate}
  Then $$|\mathfrak{g}\left(\mathcal{E}\right)|= \left| G(p) \backslash \, H(\mathbb{Z}[\zeta_{p}])  \right| \cdot \left| G(p^2) \backslash \, H(\mathbb{Z}[\zeta_{p^2}])\right| \cdot \Sigma_M.$$ 
\end{prop}
\begin{proof}

The proof is similar to the previous propositions.  It is worth noting that in this case there is always a direct summand of $M$ of type $\mathfrak{b}$ or $\mathfrak{c}$, and thus we will not use invariant (iii) of Theorem \ref{invariant sums} when counting non-isomorphic lattices. 
\end{proof}

\begin{example}
If $M$ and $\mathcal{E}$ are as in Proposition \ref{com b e c mas sem valer iv} and $p=5,$  then $|\mathfrak{g}(\mathcal{E})|=1$ if $M_{C,D}=\{0\},$ and $|\mathfrak{g}(\mathcal{E})|=2$ if $M_{C,D} \neq \{0\}.$  
\end{example}

There is one final case.

\begin{prop} \label{ultimao}
Let $\mathcal{E}=M \rtimes_{\psi_{}} G$ be a faithful semidirect product and $p \equiv \, 1 (\mbox{mod } \, 4 )$. Suppose \begin{equation} M_{}= \Z^{a} \oplus \displaystyle \bigoplus_{i=1}^{b} \, \mathfrak{b}_i \oplus  \bigoplus_{j=1}^{c} \, \mathfrak{c}_j  \oplus \bigoplus_{k=1}^{d} E(\mathfrak{b}'_k) \oplus \bigoplus_{l=1}^{e} \, E(\mathfrak{c}'_l)  \oplus \underbrace{M_B \oplus M_{C,D} \oplus M_E \oplus M_F}_{M_Q} \end{equation} 
is such that (at least) one of the following conditions is satisfied: 
\begin{enumerate}
    \item $d\geqslant 1 $ and $e \geqslant 1,$
    \item $b+c \geqslant 1$  and $(a+d+e \geqslant 1$ or $M_B \oplus M_F \neq \{0\}).$
\end{enumerate}
Then $$|\mathfrak{g}\left(\mathcal{E}\right)|= \left| G(p) \backslash \, H(\mathbb{Z}[\zeta_{p}])  \right| \cdot \left| G(p^2) \backslash \, H(\mathbb{Z}[\zeta_{p^2}])\right|.$$ 
\end{prop}
\begin{proof}
  The proof is similar to the previous propositions. The only difference is that  the conditions ensure that invariants (iii) and (iv) of Theorem \ref{invariant sums} need not be considered to distinguish isomorphism classes of $M$.
\end{proof}

\begin{example}
If $M$ and $\mathcal{E}$ are as in Proposition \ref{ultimao} and $p=5,$  then $\mathcal{E}$ is profinitely rigid. 
\end{example}

\begin{proof}[Proof of Theorem 1] %Proof of the main theorem

    The result follows from Propositions   \ref{so Cs}, \ref{Cs e Bs com absortion},  \ref{sem absorcao e sem D}, \ref{mais simples congruo a 1}, \ref{com b e c mas sem valer iv} and \ref{ultimao}.
\end{proof}

%%%%%%%%%%%%%%%%%%%%%%%%%%% references

\bibliographystyle{plain}
\bibliography{ref}

@article{heller,
  author    = {A. Heller and I. Reiner},
  title     = {{Representations of cyclic groups in rings of integers I}},
  journal   = {Annals of Mathematics},
  volume    = {76},
  number    = {1},
  year      = {1962},
  pages     = {73--92}
}

@article{jones,
  author    = {A. Jones},
  title     = {{Integral representations of cyclic $p$-groups}},
  journal   = {Anais da Academia Brasileira de Ciências},
  volume    = {54},
  number    = {1},
  year      = {1982},
  pages     = {19--22}
}

@inproceedings{reid,
  author    = {A. W. Reid},
  title     = {{Profinite rigidity}},
  booktitle = {Proceedings of the International Congress of Mathematicians, Vol. I},
  address   = {Rio de Janeiro},
  year      = {2018},
  pages     = {1191--1214}
}

@book{curtis,
  author    = {C. W. Curtis and I. Reiner},
  title     = {{Methods of Representation Theory - with applications to finite groups and orders}}, 
  volume    = {I},
  publisher = {John Wiley \& Sons},
  address   = {New York},
  year      = {1990}
}

@article{BAU74,
  author    = {G. Baumslag},
  title     = {{Residually finite groups with the same finite images}},
  journal   = {Compositio Mathematica},
  volume    = {29},
  number    = {3},
  year      = {1974},
  pages     = {249--252}
}

@article{GZ,
  author    = {F. J. Grunewald and P. A. Zalesskii},
  title     = {{Genus for groups}},
  journal   = {Journal of Algebra},
  volume    = {326},
  year      = {2011},
  pages     = {130--168}
}

@article{Ner19,
  author    = {G. J. Nery},
  title     = {{Profinite genus of fundamental groups of torus bundles}},
  journal   = {Communications in Algebra},
  volume    = {48},
  year      = {2019},
  pages     = {1567--1576}
}

@article{Ner20,
  author    = {G. J. Nery},
  title     = {{Profinite genus of fundamental groups of compact flat manifolds with holonomy group of prime order}},
  journal   = {Journal of Group Theory},
  volume    = {24},
  number    = {6},
  year      = {2021},
  pages     = {1135--1148}
}

@article{Ner24,
  author    = {G. J. Nery},
  title     = {{Profinite genus of fundamental groups of compact flat manifolds with cyclic holonomy group of square-free order}},
  journal   = {Forum Mathematicum},
  volume    = {36},
  number    = {4},
  year      = {2024},
  pages     = {881--895}
}

@article{reiner1978,
  author    = {I. Reiner},
  title     = {{Invariants of integral representations}},
  journal   = {Pacific Journal of Mathematics},
  volume    = {78},
  number    = {2},
  year      = {1978},
  pages     = {467--501}
}

@book{neukirch,
  author    = {J. Neukirch},
  title     = {{Algebraic Number Theory}},
  publisher = {Springer},
  address   = {Berlin-Heidelberg New York},
  year      = {1999}
}

@book{wash,
  author    = {L. C. Washington},
  title     = {{Introduction to cyclotomic fields}},
  publisher = {Springer-Verlag},
  address   = {Berlin/Heidelberg},
  year      = {1982}
}

@article{Ste72,
  author    = {P. F. Stebe},
  title     = {{Conjugacy separability of groups of integer matrices}},
  journal   = {Proceedings of the American Mathematical Society},
  volume    = {32},
  number    = {1},
  year      = {1972},
  pages     = {1--7}
}

@article{gal,
  author    = {S. Galovich},
  title     = {{The class group of a cyclic $p$-group}},
  journal   = {Journal of Algebra},
  volume    = {30},
  year      = {1974},
  pages     = {368--387}
}

@article{BPZ,
  author    = {V. R. de Bessa and A. L. P. Porto and P. A. Zalesskii},
  title     = {{The profinite completion of accessible groups}},
  journal   = {Monatshefte fur Mathematik},
  volume    = {202},
  number    = {2},
  year      = {2022},
  pages     = {217--227}
}

@article{BPZ2,
  author    = {V. R. de Bessa and A. L. P. Porto and P. A. Zalesskii},
  title     = {{Profinite genus of free products with finite amalgamation}},
  journal   = {Journal of Algebra},
  volume    = {643},
  year      = {2024},
  pages     = {11--48}
}

@article{BZ,
  author    = {V. R. de Bessa and P. A. Zalesskii},
  title     = {{The genus for HNN-extensions}},
  journal   = {Mathematische Nachrichten},
  volume    = {286},
  year      = {2013},
  pages     = {817--831}
}

@article{BZG,
  author    = {V. R. de Bessa and F. Grunewald  and P. A. Zalesskii},
  title     = {{Genus for virtually surface groups and pullbacks}},
  journal   = {Manuscripta Mathematica},
  volume    = {145},
  year      = {2014},
  pages     = {221--233}
}

@article{rein0,author="  I. Reiner ", title="Integral representations of cyclic groups of prime order",journal="Proceedings of the American Mathematical Society",volume="8",pages="142-146",year="1957" 
}

@article{Sam,
  author    = {S. M. Corson and S. Hughes  and P. M\"oller and O. Varghese},
  title     = {Profinite rigidity of affine Coxeter groups},
  journal   = {Mathematische Zeitschrift },
  volume    = {311},
number    = {11},
  year      = {2025},
  pages     = {1--9}
}

@article{bridson21,
  author    = {M. R. Bridson and D. B. McReynolds and A. W. Reid and R.  Spitler},
  title     = {On the profinite rigidity of triangle groups},
  journal   = {Bull. Lond. Math. Soc. },
  volume    = {53},
number    = {6},
  year      = {2021},
  pages     = {1849--1862}
}

@article{BCR16,
  author    = {M. R. Bridson and M. D. E. Conder and A. W. Reid},
  title     = {Determining Fuchsian groups by their finite quotients},
  journal   = {Israel Journal of Mathematics},
  volume    = {214},
  year      = {2016},
  pages     = {1--41}
}

@article{Finken,
  author    = {H. Finken and J. Neub\"user and W. Plesken},
  title     = {Space groups and groups of prime-power order. II. Classification
of space groups by finite factor groups},
  journal   = {Archiv der Mathematik},
  volume    = {35},
number    = {3},
  year      = {1980},
  pages     = {203--209}
}

@article{Paolini,
  author    = {D. Carolillo and G. Paolini}, 
  title     = {Profinite rigidity of crystallographic groups arising from Lie theory},
  journal   = { 
10.48550/arXiv.2506.15494},
    year      = {2025} 
}

@article{Paolini2,
  author    = {G. Paolini and R. Sklinos}, 
  title     = {Profinite rigidity of affine Coxeter groups},
  journal   = { 
arXiv.2407.01141},
    year      = {2024} 
}

@article{Sam2,
  author    = {S. M. Corson and S. Hughes  and P. M\"oller and O. Varghese},
  title     = {Higman-Thompson groups and profinite properties
of right-angled Coxeter groups},
  journal   = {Selecta Mathematica New Ser.},
  volume={32},
  number={28},
year      = {2026},
doi = {https://doi.org/10.1007/s00029-026-01130-4
},
}

@article{popovic,
title = {Distinguishing crystallographic groups by their finite quotients},
journal = {Journal of Algebra},
volume = {565},

pages = {548-563},
year = {2021},
issn = {0021-8693},
doi = {https://doi.org/10.1016/j.jalgebra.2020.06.033},
url = {https://www.sciencedirect.com/science/article/pii/S0021869320303756},
author = {P.\ Piwek and D.\ Popović and G.\ Wilkes}
}
  
\end{document}